\AddToHook{package/hyperref/before}{\RequirePackage{float}}
\PassOptionsToPackage{hyperfootnotes=false}{hyperref}
\documentclass[a4paper,fleqn]{cas-sc}

\usepackage[authoryear,longnamesfirst]{natbib}
\usepackage{mathtools}
\usepackage{algorithm}
\usepackage{algpseudocode}
\usepackage{placeins}
\usepackage[caption=false]{subfig}
\ExplSyntaxOn
\cs_new_eq:NN \qbe_printorcid: \printorcid
\RenewDocumentCommand \printorcid {}
  { \seq_if_empty:NF \g_stm_orcid_seq { \qbe_printorcid: } }
\ExplSyntaxOff

\DeclareCaptionLabelFormat{figurepanel}{\thefigure#2}
\numberwithin{algorithm}{section}
\algrenewcommand{\algorithmicrequire}{\textbf{Input:}}

\def\tsc#1{\csdef{#1}{\textsc{\lowercase{#1}}\xspace}}
\tsc{WGM}
\tsc{QE}
\newtheorem{theorem}{Theorem}[section]
\newdefinition{definition}[theorem]{Definition}

\newdefinition{remark}[theorem]{Remark}
\newproof{proof}{Proof}

\providecommand{\Kn}{\mathrm{Kn}}
\providecommand{\dif}{\mathrm{d}}
\providecommand{\Lq}{\widetilde{\mathcal{L}}_q}

\newcommand{\bm}[1]{\boldsymbol{#1}}
\newcommand{\pdfFrac}[2]{\frac{\partial #1}{\partial #2}}

\begin{document}
\let\WriteBookmarks\relax
\def\floatpagepagefraction{1}
\def\textpagefraction{.001}

\shorttitle{Effective Linear Approximation of the Quantum Boltzmann Operator}    

\shortauthors{Hu et al.}  

\title [mode = title]{A Structure-Preserving Effective Linear Approximation of the Quantum Boltzmann Collision Operator}  




\author[1]{Shuang Hu}[style=chinese]%[<options>]

%\cormark[1]


\ead{hsmathna@gmail.com}


\credit{Methodology, Formal analysis, Software, Visualization, Writing -- original draft}

\affiliation[1]{organization={Zhejiang University},
            addressline={866 Yuhangtang Road}, 
            city={Hangzhou},
%          citysep={}, % Uncomment if no comma needed between city and postcode
            postcode={310058}, 
            state={Zhejiang},
            country={China}}

\author[2]{Bo Lin}[style=chinese]%[]

\cormark[1]

\cortext[1]{Corresponding author}

% Footnote of the second author

% Email id of the second author
\ead{linbo@neps.hrl.ac.cn}


\credit{Validation, Writing -- review \& editing, Resources}

\affiliation[2]{organization={Beijing Huairou Laboratory},
            addressline={8 Yangyan East 1st Road}, 
            city={Beijing},
%          citysep={}, % Uncomment if no comma needed between city and postcode
            postcode={101400}, 
            %state={},
            country={China}}

\author[3]{Zhenning Cai}[style=chinese]

%\cormark[1]

% \cortext[1]{Corresponding author}

% Email id of the second author
\ead{matcz@nus.edu.sg}

\credit{Supervision, Conceptualization, Project administration, Writing -- review \& editing}

\affiliation[3]{organization={National University of Singapore},
            addressline={Level 4, Block S17, 10 Lower Kent Ridge Road}, 
            city={Singapore},
%          citysep={}, % Uncomment if no comma needed between city and postcode
            postcode={119076}, 
            %state={},
            country={Singapore}}



\begin{abstract}
We construct a structure-preserving effective linear approximation of the quantum Boltzmann collision operator for Bose and Fermi gases in the normal phase.
% An equilibrium-weighted radial projection reduces the four-velocity quantum factor to a function of relative speed.
For variable hard-sphere kernels, the resulting Fourier spectral evaluation costs $O(N_v^4\log N_v)$, with $N_v$ velocity nodes per direction and no angular quadrature factor.
The continuous operator conserves mass, momentum and energy and satisfies a linearized $H$-theorem.
We also use the operator at local quantum equilibrium as an approximate collision Jacobian in a modified Newton solver that retains the full nonlinear residual.
Numerical tests show smaller errors than the quantum BGK model in homogeneous relaxation and fewer full collision evaluations for steady Couette and Fourier flows than the same Newton solver using a quantum BGK Jacobian.
\nocite{*}
\end{abstract}


% \begin{highlights}
% \item A conservative linear quantum collision approximation obeys a linearized H-theorem.
% \item Fast spectral evaluation avoids angular quadrature for variable hard-sphere kernels.
% \item A modified Newton solver uses an approximate collision Jacobian for steady flows.
% \end{highlights}

% Keywords

\maketitle

% Main text
\section{Introduction}

% Introduction outline: 9 planned paragraphs.
% Reference introductions read:
% Yin--Cai--Wang, SISC 48(1), B83--B111 (2026), DOI 10.1137/25M1737791;
% introduction available at https://arxiv.org/html/2503.09580v2#S1.
% Wu, JCP 399, 108924 (2019), DOI 10.1016/j.jcp.2019.108924;
% accepted manuscript, Section 1, pp. 2--3:
% https://strathprints.strath.ac.uk/69591/1/Wu_JCP2019_A_fast_spectral_method_for_the_Uehling_Uhlenbeck_equation.pdf
% Organization: use Wu's quantum-kinetic motivation and Yin's progression from
% collision evaluation to reduced models, fast linearization and steady solvers.
% The notes below are an outline, not draft prose; each citation is assigned a role.
% All citation keys are present in cas-refs.bib. Move the planned citations into
% the corresponding sentences when drafting the Introduction.

% Para 1: Physical motivation and scope.
% 目的：说明为什么量子气体需要动理学描述，并限定本文研究对象。
% 内容：从超冷 Bose/Fermi 气体实验引入量子统计对碰撞弛豫的影响；
% 介绍 Uehling--Uhlenbeck 方程（本文统一称 QBE），指出 Bose enhancement
% 与 Pauli blocking 区别于经典碰撞；落到正常相、单组分量子气体的数值问题。
% 引文落点：实验背景 \citep{anderson1995observation,demarco1999onset}；
% 碰撞阻塞的直接实验依据 \citep{demarco2001pauli}；
% QBE 与量子输运的奠基工作 \citep{uehling1933transport}。
% 范围：BEC 实验仅作背景；本文 Bose 情形为 0<z<1，不声称处理凝聚体。
% 过渡：可靠的非平衡预测需要可负担的碰撞计算。

The realization of Bose--Einstein condensation and Fermi degeneracy in dilute atomic gases has motivated the study of quantum effects on collisional relaxation and transport~\citep{anderson1995observation,demarco1999onset}.
In the normal phase, these processes can be described by the quantum Boltzmann equation (QBE), also known as the Uehling--Uhlenbeck equation~\citep{uehling1933transport}.
Its collision operator accounts for the occupation of the final states, which enhances scattering for bosons and suppresses it for fermions.
The latter effect, Pauli blocking, has been observed in the thermal relaxation of a degenerate Fermi gas~\citep{demarco2001pauli}.
Here we focus on efficient numerical methods for the QBE of single-species Bose and Fermi gases.

% Para 2: Numerical obstacles and early computational routes.
% 目的：把物理需求转化为具体数值困难，并交代谱方法之外的相关路线。
% 内容：高维相空间、非局域碰撞积分及占据数因子使计算昂贵；
% 介绍量子 DSMC 对分布采样的需求，再提能量空间的约化确定性方法。
% 引文落点：量子 DSMC 及采样/速度离散问题 \citep{garcia2003direct}；
% 能量空间中的保守、熵相容算法 \citep{markowich2005fast}。
% 范围：后者依赖能量空间约化，不能表述为任意各向异性速度分布的通用求解器；
% 不把随机法概括为普遍低效，而以本文需要确定性碰撞响应作为转向谱方法的理由。
The QBE evolves a distribution $f(t,\bm{x},\bm{v})$ over a six-dimensional phase space, with $\bm{x},\bm{v}\in\mathbb{R}^3$.
With $N_v$ nodes per velocity direction and $M$ quadrature nodes on $\mathbb{S}^2$, direct evaluation of the collision operator involves $O(MN_v^6)$ quadrature terms per spatial point.
Direct simulation Monte Carlo (DSMC) provides a particle method for the QBE, but collision sampling requires estimates of the final-state occupations.
In the velocity-cell formulation of \citet{garcia2003direct}, these estimates are obtained from particle counts, so their accuracy depends jointly on the particle number and velocity-space resolution.
For the boson equation, \citet{markowich2005fast} developed fast deterministic schemes under an ergodic reduction to energy space; their semidiscrete form conserves mass and energy and satisfies a discrete entropy inequality.
Transport problems involving shear stress or heat flux, however, require resolving the velocity anisotropy excluded by this reduction.
Fourier spectral methods address this need by discretizing the collision operator in the full velocity space.

% Para 3: Classical Fourier spectral foundations.
% 目的：说明本文采用 Fourier 谱方法的基础及其计算优势。
% 内容：从 Fourier--Galerkin 碰撞离散，讲到利用卷积和 FFT 的快速算法，
% 再到一般碰撞核；简要区分谱精度与离散矩守恒，避免把两者混为一谈。
% 引文落点：早期 Fourier 谱框架 \citep{pareschi2000numerical}；
% 快速卷积分解 \citep{mouhot2006fast}；
% 一般碰撞核快速谱方法 \citep{gamba2017fast}；
% 以约束修正恢复离散守恒的先例 \citep{gamba2009spectral}。
% 过渡：这些工具为量子碰撞计算提供基础，但量子项还需专门处理。
Fast spectral methods for the classical Boltzmann collision operator provide the foundation for their quantum extensions.
\citet{pareschi2000numerical} developed a Fourier--Galerkin discretization of the classical Boltzmann collision operator with spectral accuracy for smooth distributions; its direct evaluation costs $O(N_v^6)$ in three dimensions.
For kernels admitting a decoupling in the Carleman representation, including hard spheres, \citet{mouhot2006fast} reduced collision evaluation to convolutions computed by the fast Fourier transform (FFT), with cost $O(MN_v^3\log N_v)$.
\citet{gamba2017fast} treated more general collision kernels through a quadrature-based decomposition with cost $O(MN_v^4\log N_v)$.
Spectral accuracy does not ensure exact discrete conservation; \citet{gamba2009spectral} enforced mass, momentum and energy conservation for elastic collisions through a constrained least-squares correction.
Extending this framework to the QBE requires further decompositions for efficient evaluation of the cubic terms introduced by the occupation factors.

% Para 4: Full quantum collision solvers and the remaining cost.
% 目的：建立完整 QBE 确定性求解的文献脉络，并准确界定计算瓶颈。
% 内容：概述量子 Fourier 离散、三次碰撞项的快速分解、一般截面及混合气体扩展；
% 补充近年并行谱算法与长期弛豫研究，随后说明本文 FullQBE 的参考框架。
% 引文落点：量子谱离散与刚性碰撞处理 \citep{filbet2010numerical}；
% 量子碰撞快速分解 \citep{hu2012fast}；
% 量子混合气体、一般微分截面和输运计算 \citep{wu2019fast}；
% 并行 Fourier 谱方法、速度重标度及凝聚/饱和的数值研究
% \citep{mouton2023deterministic}。
% 复杂度：对于一般碰撞核，Wu 的量子算法与相应经典快速谱算法同阶；
% 量子三次项增加常数开销，不能据多项式次数推断渐近复杂度升阶。
% 若列公式，统一采用本文记号 O(M N_v^4 log N_v)，M 为总角向求积点数；
% Wu 的 M^{d_v-1} 对应角向规模，须先换算记号再比较。
% 过渡：减少单次完整碰撞计算的成本之外，还可改变碰撞近似或减少调用次数。
For the quantum collision operator, \citet{filbet2010numerical} developed a Fourier spectral method.
In this scheme, three cubic terms admit FFT-based evaluation at $O(MN_v^3\log N_v)$ cost in three dimensions, whereas the remaining term requires $O(MN_v^6\log N_v)$ operations.
For kernels satisfying the Carleman decoupling assumption, \citet{hu2012fast} reduced the total cost to $O(MN_v^4\log N_v)$ through an exponential decomposition.
\citet{wu2019fast} retained this complexity for general differential cross-sections and quantum gas mixtures with unequal masses, using radial quadrature to separate the Fourier variables.
\citet{mouton2023deterministic} combined a parallel Fourier spectral implementation with velocity rescaling to investigate homogeneous relaxation and numerically observe Bose--Einstein condensation and Fermi--Dirac saturation.
In Wu's general-kernel formulation, the quantum and classical operators have the same asymptotic evaluation cost, but the cubic terms require additional convolutions.
The resulting computational burden motivates less expensive collision models.

% Para 5: Relaxation models and time-integration strategies.
% 目的：解释便宜的碰撞模型已有成熟进展，以及本文为何仍需更细致的碰撞响应。
% 内容：简述 BGK 的单松弛机制、ES-BGK 的输运改进和量子推广；
% 再区分改变碰撞模型、替换占据因子、以及仅为处理刚性而作惩罚分裂。
% 引文落点：BGK 起点 \citep{bhatnagar1954model}；
% 经典输运系数与熵相容改进 \citep{andries2000gaussian}；
% 正常相量子 ES-BGK \citep{wu2012kinetic}；
% 用平衡分布替换占据因子的 U--U model equation \citep{yano2017fast}；
% 完整 QBE 的经典 BGK 惩罚策略（处理碰撞刚性）\citep{filbet2010numerical}；
% 量子 BGK 的 AP IMEX--Hermite 方法 \citep{li2025highly}。
% 区分：AP 时间推进不等于降低单次碰撞求积成本；量子 BGK 求解器不等于 FullQBE。
% 不以“BGK 不守恒/没有 H 定理”制造差异，也不宣称本文精确恢复全部输运系数。
% 过渡：围绕平衡态线性化提供了另一种保留碰撞响应信息的途径。

A common approach is to replace the collision integral with a relaxation term.
The Bhatnagar--Gross--Krook (BGK) model uses a local Maxwellian target~\citep{bhatnagar1954model}.
Its ellipsoidal-statistical extension (ES-BGK) uses an anisotropic Gaussian to adjust the Prandtl number while preserving entropy dissipation~\citep{andries2000gaussian}, and \citet{wu2012kinetic} extended this model to quantum gases in the normal phase.
Within DSMC, \citet{yano2017fast} reduced sampling requirements by using local quantum equilibria in the occupation factors.
To treat stiffness, \citet{filbet2010numerical} used a classical BGK penalty for the full QBE, while \citet{li2025highly} developed an asymptotic-preserving (AP) implicit--explicit Hermite spectral method for the quantum BGK equation.
However, the single relaxation rate of the quantum BGK model limits its ability to capture the QBE's nonequilibrium collision response.

% Para 6: Fast classical linearization and steady-state iteration.
% 目的：引出最直接的方法来源 Yin--Cai--Wang，并明确已有工作与本文的关系。
% 内容：先提经典线性化方程的快速谱求解，再说明 Yin 等如何利用局部 Maxwellian
% 和径向 VHS 核的结构消去角向求积因子，使代价达到 O(N_v^4 log N_v)；
% 同时指出便宜的线性算子可服务于非线性稳态迭代。
% 引文落点：经典线性化谱求解/微流应用 \citep{wu2014solving}；
% 专门的快速线性化算法及 modified Newton 先例 \citep{yin2026fast}；
% 稳态加速的另一条路线 GSIS \citep{su2020steady}，仅简短定位，
% 不将其宏观--动理学耦合机制归为近似 Jacobian，也不暗示本文使用了 GSIS。
% 过渡：需要判断这种经典快速线性化结构能否用于量子平衡态。
An alternative is to retain the first-order collision response through the Fr\'echet derivative at local equilibrium.
Using fast spectral discretization, \citet{wu2014solving} solved the classical linearized Boltzmann equation for Poiseuille and thermal creep flows.
For steady linearized flows, \citet{su2020steady} accelerated convergence with a general synthetic iterative scheme (GSIS) coupling kinetic and macroscopic equations.
More recently, \citet{yin2026fast} exploited the Maxwellian structure and radial variable hard-sphere (VHS) kernels to evaluate the linearized collision operator without angular quadrature, at $O(N_v^4\log N_v)$ cost in three dimensions.
They also used this operator as an approximate collision Jacobian in a modified Newton solver for nonlinear steady flows.
This motivates a fast, structure-preserving linearized approximation for the quantum collision operator.

% Para 7: Quantum obstruction and the proposed effective operator.
% 目的：提出本文要解决的具体技术问题，而不作未经全面证明的“首次”声明。
% 内容：结合 Section 3 的推导，说明变换后的量子平衡具有 Maxwellian 形式，
% 但精确线性化中的量子因子仍耦合四个碰撞速度，因此不能直接套用
% Yin 方法所需的特定径向/卷积分解；提出构造保结构、低成本有效算子的目标。
% 引文落点：\citep{yin2026fast,wu2019fast} 分别定位经典快速结构与量子计算基础；
% 四速度耦合的具体判断由本文 Section 3 支持，不转嫁给这两篇参考文献。
% 不把该障碍表述为“量子权重破坏旋转不变性”，也不把本文算子等同于精确 Jacobian。
% 目的：集中陈述本文的算子构造、理论性质与作为碰撞模型的验证。
% 内容：概述 equilibrium-weighted radial projection 如何恢复可计算的径向权重，
% 并说明质量、动量、能量守恒、linearized H-theorem 及平衡切空间的保持；
% 接着给出 O(N_v^4 log N_v) 实现及均匀弛豫中的精度--成本比较。
% 引文落点：快速谱实现的来源 \citep{yin2026fast}；
% 新投影、结构证明与实验结论引用本文 Sections 3 和 5，无需借外部引文支撑原创结论。
% 范围：这是有效近似而非精确量子线性化；说明小 hbar 下更好的算子一致性
% 是本文测试结果，不写成全参数误差定理。连续结构与离散矩修正分开表述；
% linearized H-theorem 不升级为完整非线性 H 定理或输运系数精确性。

Efficient Hermite expansions of quantum equilibria~\citep{li2025highly} do not by themselves yield a fast evaluation of the linearized collision operator.
The exact quantum linearization contains occupation factors coupling all four collision velocities, so it does not directly inherit the convolution structure used by \citet{yin2026fast}.
We first change variables to map the quantum equilibrium to a Maxwellian.
Inspired by weighted model reduction for thin-gap flows~\citep{ding2026thin}, we then approximate the remaining quantum factor by an equilibrium-weighted radial projection.
The resulting effective operator can be evaluated at $O(N_v^4\log N_v)$ cost.
Although the projection introduces modeling error, the continuous operator conserves mass, momentum and energy and satisfies a linearized H-theorem.

% Para 8: Use in a full-residual solver and numerical evidence.
% 目的：明确本文第二种用途，区分模型替代与求解完整稳态 QBE。
% 内容：介绍局部量子平衡处的有效算子作为 modified Newton 的碰撞 Jacobian 近似，
% 完整 Q_q 仍用于非线性残量；因此残量为零时求解的仍是原离散稳态方程。
% 简述 Bose/Fermi Couette 与 Fourier 流验证，以及相较 BGK Jacobian
% 减少完整碰撞调用的结果；分别报告单次有效算子速度与整套求解器收益。
% 引文落点：modified Newton 思路的直接先例 \citep{yin2026fast}；
% Couette/Fourier 作为非平衡输运基准的经典背景
% \citep{montanero2000monte,gallis2006normal}。
% 后两篇不用于支持量子统计效应；本文的量子流动结果来自 Section 5。
% 不声称全局/二次收敛保证，也不把便宜 Jacobian 表述为消除了 full residual 的成本。
Our operator tests show close directional agreement between the effective and exact linearized collision responses.
Following \citet{yin2026fast}, we use the effective operator at local quantum equilibrium as an approximate collision Jacobian in a modified Newton solver.
The nonlinear residual retains the full collision operator, so its zeros solve the original discrete steady QBE.
We consider Bose and Fermi gases in Couette and Fourier flows, standard benchmarks in classical kinetic theory~\citep{montanero2000monte,gallis2006normal}.
In these flow tests, the effective Jacobian reduces outer iterations and full collision evaluations compared with the same solver using a quantum BGK Jacobian.

% Para 9: Organization of the paper.
% 内容：Section 2 给出 QBE、量子平衡与完整碰撞谱离散；
% Section 3 构造并分析有效线性算子；Section 4 介绍 modified Newton 稳态求解；
% Section 5 验证算子、均匀弛豫及 Couette/Fourier 流；Section 6 总结与展望。
% 引文：无需新增外部引用，使用正文现有 section labels 或明确节号。
The rest of the paper is organized as follows. Section~\ref{sec:preliminaries} reviews the QBE, quantum equilibrium and fast spectral collision evaluation.
Section~\ref{sec:linear-approx} constructs and analyzes our linear approximated operator.
Section~\ref{sec:newton} develops the modified Newton solver for the steady QBE.
Section~\ref{sec:numerical} presents numerical tests of operator accuracy, homogeneous relaxation, and steady Couette and Fourier flows.
Section~\ref{sec:conclusion} summarizes the results and outlines future work.

\section{Preliminaries}\label{sec:preliminaries}
In this section, we recall the quantum Boltzmann equation and its basic properties, together with the equilibrium distributions for Bose and Fermi gases. Then, we outline the fast spectral method (FSM) for evaluating the quantum collision operator.
\subsection{The Quantum Boltzmann Equation}
Let $f(t,\bm{x},\bm{v})\ge0$ denote the phase space distribution at time $t$, position $\bm{x}$, and velocity $\bm{v}\in\mathbb{R}^3$. The quantum Boltzmann equation reads
\begin{equation}
    \label{eq:QBE}
    \pdfFrac{f}{t}+\bm{v}\cdot\nabla_{\bm{x}} f=\frac{1}{\Kn}\mathcal{Q}_{q}(f),
\end{equation}
where $\Kn$ is the Knudsen number and
\begin{equation}
\label{eq:QBE_Collision}
\begin{aligned}
\mathcal{Q}_q(f)(\bm{v})
&=
\int_{\mathbb{R}^{3}} \int_{\mathbb{S}^{2}}
B(\bm{v}-\bm{v}_*,\bm{\sigma})
\Big[
f'f_*'(1+\tau \theta_0 f)(1+\tau \theta_0 f_*) 
-
ff_*(1+\tau\theta_0 f')(1+\tau \theta_0 f_*')
\Big]
\,\dif\bm{\sigma}\,\dif\bm{v}_* .
\end{aligned}
\end{equation}
The nonnegative collision kernel is written as $B(\bm{v}-\bm{v}_*,\bm{\sigma})=B(|\bm{v}-\bm{v}_*|,\cos\vartheta)$, where $\bm{\sigma}\in\mathbb{S}^2$ and $\cos\vartheta\coloneqq \frac{(\bm{v}-\bm{v}_{*})\cdot\bm{\sigma}}{|\bm{v}-\bm{v}_{*}|}$. Here $\theta_0=\hbar^3$, with $\hbar$ the rescaled Planck constant, and $\tau=1$ for Bose gases and $\tau=-1$ for Fermi gases. In the Fermi case, $f\le\theta_0^{-1}$. The shorthand $f,f_{*},f',f'_{*}$ denotes $f(t,\bm{x},\bm{v})$, $f(t,\bm{x},\bm{v}_*)$, $f(t,\bm{x},\bm{v}')$, and $f(t,\bm{x},\bm{v}'_{*})$, respectively, with post-collision velocities
\begin{equation}
    \label{eq:modified_vel}
    \bm{v}'\coloneqq \frac{\bm{v}+\bm{v}_{*}}{2}+\frac{|\bm{v}-\bm{v}_*|}{2}\bm{\sigma},\quad\bm{v}_{*}'\coloneqq \frac{\bm{v}+\bm{v}_{*}}{2}-\frac{|\bm{v}-\bm{v}_{*}|}{2}\bm{\sigma}.
\end{equation}

For sufficiently regular admissible $f$, $\mathcal{Q}_q$ conserves mass, momentum, and energy \citep{filbet2010numerical}:
\begin{equation}
    \label{eq:conservation_Qq}
    \int_{\mathbb{R}^3}\mathcal{Q}_q(f)(\bm{v})\dif\bm{v}=0,\quad\int_{\mathbb{R}^3}\bm{v}\mathcal{Q}_q(f)(\bm{v})\dif\bm{v}=\bm{0},\quad\int_{\mathbb{R}^3}|\bm{v}|^2\mathcal{Q}_q(f)(\bm{v})\dif\bm{v}=0,
\end{equation}
and satisfies Boltzmann's H-theorem:
\begin{equation}
    \label{eq:H_Thm}
    \int_{\mathbb{R}^3}\log\left(\frac{f}{1+\tau\theta_0f}\right)\mathcal{Q}_q(f)\dif\bm{v}\le 0.
\end{equation}
The density $\rho$, bulk velocity $\bm{u}$, kinetic temperature $\theta$,
pressure tensor $\mathbb{P}$, and heat flux $\bm{q}$ are defined by
\begin{equation}
    \label{eq:momentums}
    \begin{aligned}
        &\rho=\int_{\mathbb{R}^3}f\dif\bm{v},\quad
        \rho\bm{u}=\int_{\mathbb{R}^3}\bm{v}f\dif\bm{v},\quad
        \rho\theta=\frac{1}{3}\int_{\mathbb{R}^3}
        |\bm{v}-\bm{u}|^2f\dif\bm{v},\\
        &\mathbb{P}=\int_{\mathbb{R}^3}
        (\bm{v}-\bm{u})\otimes(\bm{v}-\bm{u})f\dif\bm{v},\quad
        \bm{q}=\frac{1}{2}\int_{\mathbb{R}^3}
        (\bm{v}-\bm{u})|\bm{v}-\bm{u}|^2f\dif\bm{v}.
    \end{aligned}
\end{equation}
Here $\theta$ measures the energy of random particle motion relative
to $\bm{u}$, while $\mathbb{P}$ and $\bm{q}$ describe the associated
momentum and energy transport. The scalar pressure is
$p=\rho\theta=\frac{1}{3}\operatorname{tr}\mathbb{P}$.
\subsection{Quantum Equilibrium}
The local quantum equilibrium $\mathcal{M}_q$ satisfies $\mathcal{Q}_q(\mathcal{M}_q)(\bm{v})\equiv0$. It has the form \citep{filbet2010numerical}
\begin{equation}
    \label{eq:local_equilibrium_QBE}
    \mathcal{M}_{q}(t,\bm{x},\bm{v})\coloneqq
    \frac{1}{\theta_0}\frac{z}{e^{\frac{|\bm{v}-\bm{u}|^2}{2T}}-\tau z},
\end{equation}
where $z$ is the fugacity and $T>0$ the thermodynamic temperature, which generally differs from the kinetic temperature $\theta=\frac{p}{\rho}$. The admissible fugacity range is $0<z\le1$ for Bose gases and $0<z<\infty$ for Fermi gases.
The Bose endpoint $z=1$ marks the onset of Bose--Einstein condensation \citep{filbet2010numerical}.
Throughout this work, we restrict the Bose case to $0<z<1$.
The parameters $(z,T)$ are determined by matching the moments $(\rho,\theta)$. For real $s>1$ and $y\le1$, define the polylogarithm function by
\begin{equation}
    \label{eq:polylogarithm}
    \operatorname{Li}_s(y)
    =\frac{y}{\Gamma(s)}
    \int_0^\infty\frac{r^{s-1}}{e^r-y}\,\dif r,
\end{equation}
where $\Gamma$ is the gamma function. Moment matching gives the $2\times2$ nonlinear system for $(z,T)$ \citep{li2025highly}:
\begin{equation}
    \label{eq:fugality_and_temperature}
    \frac{\bigl[\tau\operatorname{Li}_{\frac{3}{2}}(\tau z)\bigr]^{\frac{5}{2}}}
         {\bigl[\tau\operatorname{Li}_{\frac{5}{2}}(\tau z)\bigr]^{\frac{3}{2}}}
    =\frac{\theta_0\rho}{(2\pi\theta)^{\frac{3}{2}}},
    \qquad
    T=\theta\,
    \frac{\operatorname{Li}_{\frac{3}{2}}(\tau z)}
         {\operatorname{Li}_{\frac{5}{2}}(\tau z)}.
\end{equation}
An efficient algorithm for evaluating these polylogarithms is given in \citet{li2025highly}.

Define
\begin{equation}
    \label{eq:Quantum_equil_to_Classical_Maxwellian}
    \mathcal{N}(\bm{v};z,T,\bm{u})\coloneqq
    \frac{\mathcal{M}_q}{1+\tau\theta_0\mathcal{M}_q}
    =\frac{z}{\theta_0}e^{-\frac{|\bm{v}-\bm{u}|^2}{2T}},
\end{equation}
and the shorthand $\mathcal{N},\mathcal{N}',\mathcal{N}_*,\mathcal{N}_*'$ denotes $\mathcal{N}(\bm{v}),\mathcal{N}(\bm{v}'),\mathcal{N}(\bm{v}_*),\mathcal{N}(\bm{v}_*')$, respectively.  
Conservation of momentum and energy gives the factorization
\begin{equation}
    \label{eq:Maxwellian_Separate}
        \mathcal{N}\mathcal{N}_*
        =\mathcal{N}'\mathcal{N}_*'
        =\mathcal{N}^{g}(\bm{v}-\bm{v}_*)
         \mathcal{N}^{c}\!\left(\frac{\bm{v}+\bm{v}_*}{2}\right),\quad
        \mathcal{N}^{g}(\bm{g})
        =e^{-\frac{|\bm{g}|^2}{4T}},\quad
        \mathcal{N}^{c}(\bm{c})
        =\left(\frac{z}{\theta_0}\right)^2
         e^{-\frac{|\bm{c}-\bm{u}|^2}{T}}.
\end{equation}
This factorization is used in the effective linear approximation below. 

\begin{remark}
The variables $\bm{g}=\bm{v}-\bm{v}_*$ and
$\bm{c}=\frac{\bm{v}+\bm{v}_*}{2}$ are the relative and center-of-mass
velocities of the colliding pair.
In the transformed product $\mathcal{N}\mathcal{N}_*$,
$\mathcal{N}^{g}$ weights relative motion, while $\mathcal{N}^{c}$
weights center-of-mass motion relative to $\bm{u}$.
Elastic collisions preserve $|\bm{g}|$ and $\bm{c}$, so 
$\mathcal{N}^{g}$ and $\mathcal{N}^{c}$ are unchanged after collision.
% is unchanged.
\end{remark}
\subsection{Fast Spectral Evaluation}
\label{sec:FSM_Direct}
Following \citet{filbet2010numerical}, we write the quantum collision operator \eqref{eq:QBE_Collision} as
\begin{equation}
    \label{eq:split_Quantum_collision}
    \mathcal{Q}_{q}=\mathcal{Q}_{c}+\tau\theta_0(\mathcal{Q}_1+\mathcal{Q}_2-\mathcal{Q}_3-\mathcal{Q}_4),
\end{equation}
where the classical Boltzmann collision operator is
\begin{equation}
    \label{eq:classical_Boltzmann}
    \mathcal{Q}_c(f)(\bm{v})=\int_{\mathbb{R}^3}\int_{\mathbb{S}^2}B(|\bm{v}-\bm{v}_*|,\cos\vartheta)(f'f_{*}'-ff_*)\dif\bm{\sigma}\dif\bm{v}_{*},
\end{equation}
and the cubic terms are
\begin{equation}
\label{eq:origin_form_Q}
\left\{
\begin{aligned}
\mathcal{Q}_1(f)(\bm{v})
&=
\int_{\mathbb{R}^{3}}
\int_{\mathbb{S}^{2}}
B(|\bm{v}-\bm{v}_*|,\cos\vartheta)
f'f_*'f_*
\,\mathrm{d}\bm{\sigma}\,\mathrm{d}\bm{v}_*,\quad
&\mathcal{Q}_2(f)(\bm{v})
=\int_{\mathbb{R}^{3}}
\int_{\mathbb{S}^{2}}
B(|\bm{v}-\bm{v}_*|,\cos\vartheta)
f'f_*'f
\,\mathrm{d}\bm{\sigma}\,\mathrm{d}\bm{v}_*,\\
\mathcal{Q}_3(f)(\bm{v})
&=
\int_{\mathbb{R}^{3}}
\int_{\mathbb{S}^{2}}
B(|\bm{v}-\bm{v}_*|,\cos\vartheta)
ff_*f'
\,\mathrm{d}\bm{\sigma}\,\mathrm{d}\bm{v}_*,\quad
&\mathcal{Q}_4(f)(\bm{v})=
\int_{\mathbb{R}^{3}}
\int_{\mathbb{S}^{2}}
B(|\bm{v}-\bm{v}_*|,\cos\vartheta)
ff_*f_*'
\,\mathrm{d}\bm{\sigma}\,\mathrm{d}\bm{v}_* .
\end{aligned}
\right.
\end{equation}
At fixed time and position, let $\mathcal{B}_r=\{\bm{w}\in\mathbb{R}^3:|\bm{w}|\le r\}$ and choose $S>0$ so that $f$ is negligible outside $\mathcal{B}_S$.
We restrict $f$ to the velocity domain $\mathcal{D}_L=[-L,L]^3$ and extend it with period $2L$ in each direction.
The uniform velocity grid has $N_v$ points per direction and spacing $\Delta v=2L/N_v$.

We denote $R$ as the cutoff in relative speed $|\bm{v}-\bm{v}_*|$, rather than the support radius of $f$.
Under this effective compact-support assumption, we choose $R$ and $L$
to control aliasing due to periodization \citep{yin2026fast}:
\begin{equation}
    \label{eq:velocity_truncation}
    R=2S,\qquad
    L\ge\frac{3+\sqrt{2}}{2}S
      =\frac{3+\sqrt{2}}{4}R.
\end{equation}

In the Carleman representation \citep{mouhot2006fast}, $\bm{y}=\bm{v}'-\bm{v}$ and $\bm{z}=\bm{v}_*'-\bm{v}$ are velocity increments.
They satisfy $\bm{y}\cdot\bm{z}=0$ and $|\bm{v}-\bm{v}_*|^2=|\bm{y}|^2+|\bm{z}|^2$.
The Carleman kernel before truncation reads
\begin{equation}
\label{eq:tildeB}
\widetilde{B}(\bm{y},\bm{z})
=
\frac{4}
{\left(|\bm{y}|^2+|\bm{z}|^2\right)^{\frac{1}{2}}}
B\left(
\sqrt{|\bm{y}|^2+|\bm{z}|^2},
1-\frac{2|\bm{y}|^2}{|\bm{y}|^2+|\bm{z}|^2}
\right)
=
\widetilde{B}(|\bm{y}|,|\bm{z}|).
\end{equation}
We impose the relative-speed cutoff by
\begin{equation}
    \label{eq:Carleman_cutoff}
    \widetilde{B}_R(\bm{y},\bm{z})
    =\widetilde{B}(\bm{y},\bm{z})
     \mathbf{1}_{\{|\bm{y}|^2+|\bm{z}|^2\le R^2\}},
\end{equation}
where $\mathbf{1}$ denotes the indicator function.
The joint cutoff makes the Carleman formulation of \citet{wu2019fast}
consistent with the relative-speed truncation used by \citet{yin2026fast}.
The resulting truncated, periodized cubic operators are
\begin{equation}
\label{eq:Carleman_rep}
\left\{
\begin{aligned}
\mathcal{Q}_1^R(f)(\bm{v})
&=
\int_{\mathcal{B}_R}
\int_{\mathcal{B}_R}
\widetilde{B}_R(\bm{y},\bm{z})
\delta(\bm{y}\cdot\bm{z})
f(\bm{v}+\bm{y})
f(\bm{v}+\bm{z})
f(\bm{v}+\bm{y}+\bm{z})
\,\mathrm{d}\bm{y}\,\mathrm{d}\bm{z},
\\[0.8em]
\mathcal{Q}_2^R(f)(\bm{v})
&=
\int_{\mathcal{B}_R}
\int_{\mathcal{B}_R}
\widetilde{B}_R(\bm{y},\bm{z})
\delta(\bm{y}\cdot\bm{z})
f(\bm{v}+\bm{y})
f(\bm{v}+\bm{z})
f(\bm{v})
\,\mathrm{d}\bm{y}\,\mathrm{d}\bm{z},
\\[0.8em]
\mathcal{Q}_3^R(f)(\bm{v})
&=
\int_{\mathcal{B}_R}
\int_{\mathcal{B}_R}
\widetilde{B}_R(\bm{y},\bm{z})
\delta(\bm{y}\cdot\bm{z})
f(\bm{v}+\bm{y})
f(\bm{v}+\bm{y}+\bm{z})
f(\bm{v})
\,\mathrm{d}\bm{y}\,\mathrm{d}\bm{z},
\\[0.8em]
\mathcal{Q}_4^R(f)(\bm{v})
&=
\int_{\mathcal{B}_R}
\int_{\mathcal{B}_R}
\widetilde{B}_R(\bm{y},\bm{z})
\delta(\bm{y}\cdot\bm{z})
f(\bm{v}+\bm{z})
f(\bm{v}+\bm{y}+\bm{z})
f(\bm{v})
\,\mathrm{d}\bm{y}\,\mathrm{d}\bm{z},
\end{aligned}
\right.
\end{equation}
Let $N_v$ be even and set $N=N_v/2$. We approximate $f$ by
\begin{equation}
\label{eq:Fourier_expansion}
    f_N(\boldsymbol{v})
=
\sum_{\boldsymbol{k}=-N}^{N-1}
\hat{f}_{\boldsymbol{k}}
e^{i\pi \boldsymbol{k}\cdot\boldsymbol{v}/L},
\qquad
\hat{f}_{\boldsymbol{k}}
=
\left(\frac{1}{2N}\right)^3
\sum_{\boldsymbol{\ell}=-N}^{N-1}
f\left(\frac{\boldsymbol{\ell}L}{N}\right)
e^{-i\pi \boldsymbol{k}\cdot\boldsymbol{\ell}/N},
\end{equation}
where $\boldsymbol{k}=(k_1,k_2,k_3)$ and
$\sum_{\boldsymbol{k}=-N}^{N-1}
=\sum_{k_1=-N}^{N-1}\sum_{k_2=-N}^{N-1}\sum_{k_3=-N}^{N-1}.$
Each Fourier index ranges from $-N$ to $N-1$, so the Nyquist frequency is represented only once.
For $i=1,\ldots,4$, the corresponding Fourier expansion of $\mathcal{Q}_i^R(f_N)$ is
\begin{equation}
\mathcal{Q}_{i,N}^R(f_N)(\bm{v})
=\sum_{\bm{k}=-N}^{N-1}
  \widehat{\mathcal{Q}}_{i,\bm{k}}e^{i\pi\bm{k}\cdot\bm{v}/L},
\qquad
\widehat{\mathcal{Q}}_{i,\bm{k}}
=\frac{1}{(2L)^3}\int_{\mathcal{D}_L}\mathcal{Q}_i^R(f_N)(\bm{v})
  e^{-i\pi\bm{k}\cdot\bm{v}/L}\,\dif\bm{v}.
\end{equation}
Substituting \eqref{eq:Fourier_expansion} gives
\begin{equation}
\label{eq:Fourier_modes_QBE}
\left\{
\begin{alignedat}{2}
\widehat{\mathcal{Q}}_{1,\bm{k}}
&=
\sum_{\substack{
\bm{l},\bm{m},\bm{n}=-N \\
\bm{l}+\bm{m}+\bm{n}=\bm{k}
}}^{N-1}
\beta(\bm{l}+\bm{n},\bm{m}+\bm{n})
\widehat{f}_{\bm{l}}\widehat{f}_{\bm{m}}\widehat{f}_{\bm{n}},
\qquad
&
\widehat{\mathcal{Q}}_{2,\bm{k}}
&=
\sum_{\substack{
\bm{l},\bm{m},\bm{n}=-N \\
\bm{l}+\bm{m}+\bm{n}=\bm{k}
}}^{N-1}
\beta(\bm{l},\bm{m})
\widehat{f}_{\bm{l}}\widehat{f}_{\bm{m}}\widehat{f}_{\bm{n}},
\\[1.2em]
\widehat{\mathcal{Q}}_{3,\bm{k}}
&=
\sum_{\substack{
\bm{l},\bm{m},\bm{n}=-N \\
\bm{l}+\bm{m}+\bm{n}=\bm{k}
}}^{N-1}
\beta(\bm{l}+\bm{m},\bm{m})
\widehat{f}_{\bm{l}}\widehat{f}_{\bm{m}}\widehat{f}_{\bm{n}},
\qquad
&
\widehat{\mathcal{Q}}_{4,\bm{k}}
&=
\sum_{\substack{
\bm{l},\bm{m},\bm{n}=-N \\
\bm{l}+\bm{m}+\bm{n}=\bm{k}
}}^{N-1}
\beta(\bm{m},\bm{l}+\bm{m})
\widehat{f}_{\bm{l}}\widehat{f}_{\bm{m}}\widehat{f}_{\bm{n}} ,
\end{alignedat}
\right.
\end{equation}
where the kernel modes are
\begin{equation}
\label{eq:beta}
\beta(\bm{l},\bm{m})
=
\int_{\mathcal{B}_R}
\int_{\mathcal{B}_R}
\widetilde{B}_{R}(|\bm{y}|,|\bm{z}|)
\delta(\bm{y}\cdot\bm{z})
e^{i\frac{\pi}{L}\bm{l}\cdot\bm{y}}
e^{i\frac{\pi}{L}\bm{m}\cdot\bm{z}}
\,\mathrm{d}\bm{y}\,\mathrm{d}\bm{z}.
\end{equation}

Following \citet{wu2019fast}, we use quadrature to approximate $\beta$ in separated form and evaluate the resulting convolutions by FFTs.
For radial quadrature $(r_j,a_j)_{j=1}^{N_r}$ on $[0,R]$ and spherical quadrature $(\bm{\omega}_p,w_p)_{p=1}^{M}$ with surface measure on $\mathbb{S}^2$, setting $\bm{y}=r\bm{\omega}$ gives
\[
\begin{aligned}
\beta(\bm{l},\bm{m})
\approx\sum_{j=1}^{N_r}\sum_{p=1}^{M}
a_jw_pr_j e^{i\pi r_j\bm{l}\cdot\bm{\omega}_p/L}\,
\Phi_{jp}(\bm{m}),\quad
\Phi_{jp}(\bm{m})
=\int_{\mathcal{B}_R}
\widetilde{B}_R(r_j,|\bm{z}|)
\delta(\bm{\omega}_p\cdot\bm{z})
e^{i\pi\bm{m}\cdot\bm{z}/L}\,\dif\bm{z}.
\end{aligned}
\]
Taking $N_r=O(N_v)$ yields $O(MN_v)$ separated terms.
With $M$ Lebedev quadrature nodes on $\mathbb{S}^2$ \citep{LebedevLaikov1999} and $O(N_v)$ radial nodes, each collision evaluation requires $O(MN_v^4\log N_v)$ operations after kernel precomputation.

\begin{remark}
In this framework, the quantum and classical collision operators have the same asymptotic complexity, but the additional cubic terms typically increase the constant factor for the quantum operator.
Reducing this overhead is therefore particularly valuable for QBE computations.
\end{remark}

\section{A Structure-Preserving Effective Linear Approximation}\label{sec:linear-approx}
In this section, we construct an effective linear approximation to $\mathcal{Q}_q$ about $\mathcal{M}_q$ with three objectives:
\begin{itemize}
    \item time complexity: $O(N_v^4\log N_v)$;
    \item preservation of mass, momentum, energy and the equilibrium;
    \item accurate approximation of the Fr\'echet derivative of $\mathcal{Q}_q$ at $\mathcal{M}_q$ for small $\hbar$, and use as a Jacobian surrogate in modified Newton iterations for larger $\hbar$.
\end{itemize}
We consider the variable hard-sphere (VHS) collision kernel, which is independent of the scattering direction $\bm{\sigma}$:
\begin{equation}
    \label{eq:B_VHS}
    B(\bm{v}-\bm{v}_*,\bm{\sigma})=C_B|\bm{v}-\bm{v}_*|^{2(1-w)}.
\end{equation}
Here $C_B>0$ and $w$ are fixed kernel parameters. We first derive the exact linearization of $\mathcal{Q}_q$ and explain why its quantum factor obstructs direct application of the fast spectral method of \citet{yin2026fast}. We then apply an equilibrium-weighted radial projection to recover the convolution structure required by that method. Finally, we establish the structural properties of the resulting effective operator and describe its fast spectral discretization.

\subsection{Exact Linearization}
To expose the Maxwellian product structure in \eqref{eq:Maxwellian_Separate}, we introduce the change of variables
\begin{equation}
    \label{eq:modified_func}
    \mathcal{F}\coloneqq\frac{f}{1+\tau\theta_0 f},
    \qquad
    f=\frac{\mathcal{F}}{1-\tau\theta_0\mathcal{F}},
\end{equation}
with $f<\theta_0^{-1}$ in the Fermi case.
Retaining the notation $\mathcal{Q}_q$ for the collision operator expressed in $\mathcal{F}$, we obtain
\begin{equation}
    \label{eq:Quantum_collision_by_R}
    \mathcal{Q}_q(\mathcal{F})(\bm{v})
    =\int_{\mathbb{R}^3}\int_{\mathbb{S}^2}
    B(|\bm{v}-\bm{v}_*|)\mathcal{A}(\mathcal{F})
    \bigl(\mathcal{F}'\mathcal{F}_*'-\mathcal{F}\mathcal{F}_*\bigr)
    \dif\bm{\sigma}\dif\bm{v}_*,
\end{equation}
where the auxiliary factor is
\begin{equation}
    \label{eq:auxiliary_factor}
    \mathcal{A}(\mathcal{F})
    \coloneqq
    \frac{1}{
    (1-\tau\theta_0\mathcal{F})(1-\tau\theta_0\mathcal{F}_*)
    (1-\tau\theta_0\mathcal{F}')(1-\tau\theta_0\mathcal{F}_*')}.
\end{equation}
The equilibrium $\mathcal{M}_q$ is mapped to the Maxwellian $\mathcal{N}$ defined in \eqref{eq:Quantum_equil_to_Classical_Maxwellian}.

Fix the equilibrium $\mathcal{N}$ and consider the perturbation $\mathcal{N}+\varepsilon\eta$.
Assume that the collision integrals converge absolutely and that differentiation under the integral sign is justified.
The exact Fr\'echet derivative with respect to $\mathcal{F}$ at $\mathcal{N}$ is
\begin{equation}
    \label{eq:exact_Frechet_R}
    \begin{aligned}
    \mathrm{D}_{\mathcal{F}}\mathcal{Q}_q[\mathcal{N}](\eta)(\bm{v})
    &=\left.\frac{\dif}{\dif\varepsilon}
    \mathcal{Q}_q(\mathcal{N}+\varepsilon\eta)(\bm{v})
    \right|_{\varepsilon=0}\\
    &=\int_{\mathbb{R}^3}\int_{\mathbb{S}^2}
    B(|\bm{v}-\bm{v}_*|)\mathcal{A}(\mathcal{N})
    \bigl(\mathcal{N}_*'\eta'+\mathcal{N}'\eta_*'
    -\mathcal{N}_*\eta-\mathcal{N}\eta_*\bigr)
    \dif\bm{\sigma}\dif\bm{v}_*.
    \end{aligned}
\end{equation}
Setting $r=\eta/\mathcal{N}$ and using \eqref{eq:Maxwellian_Separate}, we obtain the equivalent quotient form
\begin{equation}
    \label{eq:linearized_collision}
    \begin{aligned}
    \mathcal{L}_q(r)(\bm{v})
    \coloneqq
    \mathrm{D}_{\mathcal{F}}\mathcal{Q}_q[\mathcal{N}](\mathcal{N}r)(\bm{v})=\int_{\mathbb{R}^3}\int_{\mathbb{S}^2}
    B(|\bm{v}-\bm{v}_*|)\mathcal{A}(\mathcal{N})
    \mathcal{N}\mathcal{N}_*
    (r'+r_*'-r-r_*)
    \dif\bm{\sigma}\dif\bm{v}_*.
    \end{aligned}
\end{equation}

\begin{remark}
To approximate the linear collision response for a given $f$ near $\mathcal{M}_q$, apply
$\mathrm{D}_{\mathcal{F}}\mathcal{Q}_q[\mathcal{N}]$ to the full transformed increment
$\mathcal{F}-\mathcal{N}$.
Equivalently, one can also evaluate $\mathcal{L}_q(r)$ with
$r=(\mathcal{F}-\mathcal{N})/\mathcal{N}$.
For $f=\mathcal{M}_q+\varepsilon g$, \eqref{eq:modified_func} gives
\begin{equation}
    \label{eq:perturbation_ratio}
    \begin{aligned}
    \frac{\mathcal{F}-\mathcal{N}}{\mathcal{N}}
    =\frac{\varepsilon g}{
    \mathcal{M}_q(1+\tau\theta_0\mathcal{M}_q)
    +\varepsilon\tau\theta_0\mathcal{M}_q g}=\frac{f-\mathcal{M}_q}{
    \mathcal{M}_q(1+\tau\theta_0\mathcal{M}_q)}
    +O(\varepsilon^2).
    \end{aligned}
\end{equation}
Thus one may instead take
$r=(f-\mathcal{M}_q)/[\mathcal{M}_q(1+\tau\theta_0\mathcal{M}_q)]$.
The two choices differ by $O(\varepsilon^2)$ and give the same first-order collision response.
\end{remark}

\begin{remark}
Although $\mathcal{N}\mathcal{N}_*$ has the factorization \eqref{eq:Maxwellian_Separate}, $\mathcal{A}(\mathcal{N})$ couples all four pre- and post-collision velocities.
In general, this factor depends on the center-of-mass velocity as well as the relative velocity and scattering direction.
The resulting operator therefore does not directly inherit the convolution structure of the classical linearized collision operator, so it does not admit an $O(N_v^4\log N_v)$ algorithm.
\end{remark}
\subsection{Equilibrium-Weighted Radial Projection}
To apply the fast spectral method of \citet{yin2026fast}, we replace the four-velocity factor $\mathcal{A}(\mathcal{N})$ in \eqref{eq:linearized_collision} by a radial factor $\widetilde{\mathcal{A}}_{\mathcal{N}}(|\bm{v}-\bm{v}_*|)$.
Its product with the VHS kernel remains radial and therefore retains the kernel structure required by that method.

For a fixed equilibrium, $\mathcal{A}(\mathcal{N})$ is independent of the perturbation $r$, which enters \eqref{eq:linearized_collision} only through $r'+r_*'-r-r_*$.
We choose the radial factor independently of $r$ as well, preserving linearity and allowing the same kernel to be used for different perturbations.
We also require the replacement to preserve the total equilibrium collision weight:
\begin{equation}
    \label{eq:radial_weight_matching}
    \begin{aligned}
    &\int_{\mathbb{R}^3}\int_{\mathbb{R}^3}\int_{\mathbb{S}^2}
    B(|\bm{v}-\bm{v}_*|)\mathcal{A}(\mathcal{N})
    \mathcal{N}\mathcal{N}_*
    \dif\bm{\sigma}\dif\bm{v}_*\dif\bm{v}=
    \int_{\mathbb{R}^3}\int_{\mathbb{R}^3}\int_{\mathbb{S}^2}
    B(|\bm{v}-\bm{v}_*|)
    \widetilde{\mathcal{A}}_{\mathcal{N}}(|\bm{v}-\bm{v}_*|)
    \mathcal{N}\mathcal{N}_*
    \dif\bm{\sigma}\dif\bm{v}_*\dif\bm{v}.
    \end{aligned}
\end{equation}
This condition matches the two mean quantum factors $\mathcal{A}(\mathcal{N})$ and $\widetilde{\mathcal{A}}(\mathcal{N})$ 
under the weight $B(|\bm{v}-\bm{v}_*|)\mathcal{N}\mathcal{N}_*$.
Based on this condition, we propose the following construction of $\widetilde{A}(\mathcal{N})$, which enforces a stronger condition that requires \eqref{eq:radial_weight_matching} to hold for any admissible radial function $B$.

\begin{theorem}[Equilibrium-weighted radial projection]\label{thm:radial_projection}
For $s=|\bm{v}-\bm{v}_*|>0$, introduce the collision coordinates
\begin{equation}
\bm{c}=\frac{\bm{v}+\bm{v}_*}{2},
\qquad
\bm{\omega}=\frac{\bm{v}-\bm{v}_*}{s},
\qquad
\bm{\sigma}=\frac{\bm{v}'-\bm{v}_*'}{s},
\end{equation}
where $\bm{c}\in\mathbb{R}^3$ and $\bm{\omega},\bm{\sigma}\in\mathbb{S}^2$.
For each $s>0$, define the equilibrium-weighted average
\begin{equation}
\label{eq:radial_projection}
\widetilde{\mathcal{A}}_{\mathcal{N}}(s)
\coloneqq
\frac{
\displaystyle
\int_{\mathbb{R}^3}\int_{\mathbb{S}^2}\int_{\mathbb{S}^2}
\mathcal{N}\mathcal{N}_*
\mathcal{A}(\mathcal{N})(\bm{c},s,\bm{\omega},\bm{\sigma})
\,\dif\bm{\sigma}\,\dif\bm{\omega}\,\dif\bm{c}
}{
\displaystyle
\int_{\mathbb{R}^3}\int_{\mathbb{S}^2}\int_{\mathbb{S}^2}
\mathcal{N}\mathcal{N}_*
\,\dif\bm{\sigma}\,\dif\bm{\omega}\,\dif\bm{c}
}.
\end{equation}
At $s=0$, all four collision velocities equal $\bm{c}$, so
\begin{equation*}
\widetilde{\mathcal{A}}_{\mathcal{N}}(0)
=
\frac{\displaystyle\int_{\mathbb{R}^3}
\mathcal{N}(\bm{c})^2[1-\tau\theta_0\mathcal{N}(\bm{c})]^{-4}\,\dif\bm{c}}
{\displaystyle\int_{\mathbb{R}^3}\mathcal{N}(\bm{c})^2\,\dif\bm{c}}.
\end{equation*}
For any measurable radial weight $F(s)$ for which the integrals below converge absolutely, this average satisfies
\begin{equation}
\label{eq:radial_moment_preservation}
\begin{aligned}
&\int_{\mathbb{R}^3}\int_{\mathbb{R}^3}\int_{\mathbb{S}^2}
F(|\bm{v}-\bm{v}_*|)\mathcal{A}(\mathcal{N})
\mathcal{N}\mathcal{N}_*
\,\dif\bm{\sigma}\,\dif\bm{v}_*\,\dif\bm{v}=
\int_{\mathbb{R}^3}\int_{\mathbb{R}^3}\int_{\mathbb{S}^2}
F(|\bm{v}-\bm{v}_*|)
\widetilde{\mathcal{A}}_{\mathcal{N}}(|\bm{v}-\bm{v}_*|)
\mathcal{N}\mathcal{N}_*
\,\dif\bm{\sigma}\,\dif\bm{v}_*\,\dif\bm{v}.
\end{aligned}
\end{equation}
The radial factor is positive, with
$\widetilde{\mathcal{A}}_{\mathcal{N}}(s)\ge1$ for Bose gases
($\tau=1$, $0<z<1$) and
$0<\widetilde{\mathcal{A}}_{\mathcal{N}}(s)\le1$ for Fermi gases
($\tau=-1$, $z>0$).
In both cases, $\widetilde{\mathcal{A}}_{\mathcal{N}}(s)\to1$
as $z\to0$ and as $s/\sqrt{T}\to\infty$, with $z$ fixed in the latter limit.
\end{theorem}

\begin{proof}
Set $\bm{g}=\bm{v}-\bm{v}_*$.
The linear transformation $(\bm{v},\bm{v}_*)\mapsto(\bm{c},\bm{g})$ has unit absolute Jacobian.
Writing $\bm{g}=s\bm{\omega}$ in spherical coordinates gives
\begin{equation}
\dif\bm{v}\,\dif\bm{v}_*
=\dif\bm{c}\,\dif\bm{g},
\qquad
\dif\bm{g}=s^2\,\dif s\,\dif\bm{\omega},
\end{equation}
where $\dif\bm{\omega}$ is surface measure on $\mathbb{S}^2$.
Multiplying \eqref{eq:radial_projection} by its denominator and integrating against $F(s)s^2\,\dif s$ yields \eqref{eq:radial_moment_preservation} after this change of variables.

Since $\mathcal{N}\mathcal{N}_*>0$, the average inherits the pointwise bounds on $\mathcal{A}(\mathcal{N})$.
For Bose gases, $0<\theta_0\mathcal{N}(\bm{v})\le z<1$ implies
$1\le\mathcal{A}(\mathcal{N})\le(1-z)^{-4}$.
For Fermi gases, every factor in the denominator of \eqref{eq:auxiliary_factor} exceeds one, so
$0<\mathcal{A}(\mathcal{N})\le1$.

For the limits, \eqref{eq:Maxwellian_Separate} shows that the common factor
$(z/\theta_0)^2e^{-s^2/(4T)}$ cancels in \eqref{eq:radial_projection}.
After the rescaling $\bm{c}=\bm{u}+\sqrt{T}\bm{y}$,
the resulting normalized product of Gaussian and surface measures is independent of $z$ and $s/\sqrt{T}$.
In these variables, $\mathcal{A}(\mathcal{N})\to1$ pointwise in either limit.
For $0<z\le z_0<1$ in the Bose case, $\mathcal{A}(\mathcal{N})\le(1-z_0)^{-4}$; in the Fermi case, $\mathcal{A}(\mathcal{N})\le1$.
Dominated convergence therefore yields both limits.
\qed
\end{proof}

\begin{remark}
For each fixed relative speed $s$, the average in \eqref{eq:radial_projection} is the weighted $L^2$-projection of $\mathcal{A}(\mathcal{N})$ onto constants in $(\bm{c},\bm{\omega},\bm{\sigma})$:
\begin{equation}
\widetilde{\mathcal{A}}_{\mathcal{N}}(s)
=
\arg\min_{a\in\mathbb{R}}
\int_{\mathbb{R}^3}\int_{\mathbb{S}^2}\int_{\mathbb{S}^2}
\mathcal{N}\mathcal{N}_*
\left|\mathcal{A}(\mathcal{N})-a\right|^2
\,\dif\bm{\sigma}\,\dif\bm{\omega}\,\dif\bm{c}.
\end{equation}
\end{remark}

For numerical evaluation, the angular symmetry of \eqref{eq:radial_projection} reduces this average to two nested one-dimensional quadratures.
\begin{theorem}
\label{thm:precompute}
Let $\lambda=s/\sqrt{T}$.
Then $\widetilde{\mathcal{A}}_{\mathcal{N}}(s)=\mathcal{A}_\tau(z,\lambda)$, where
\begin{equation}
\label{eq:easy_version_aux}
\mathcal{A}_\tau(z,\lambda)
=
\frac{1}{\sqrt{\pi}}
\int_0^\infty q^2e^{-q^2}
\left[
\int_{-1}^{1}
\frac{\dif\xi}{
\left(1-\tau z e^{-\frac{q^2+\lambda^2/4+q\lambda\xi}{2}}\right)
\left(1-\tau z e^{-\frac{q^2+\lambda^2/4-q\lambda\xi}{2}}\right)}
\right]^2\dif q.
\end{equation}
\end{theorem}

\begin{proof}
Set $\bm{c}=\bm{u}+\sqrt{T}\bm{y}$ and $q=|\bm{y}|$.
For $q>0$, write $\xi=(\bm{y}/q)\cdot\bm{\omega}$; at $q=0$, the expressions below are independent of $\xi$.
By \eqref{eq:Quantum_equil_to_Classical_Maxwellian},
\begin{equation*}
\theta_0\mathcal{N}(\bm{c}\pm s\bm{\omega}/2)
=z\exp\left(-\frac{|\bm{y}\pm\lambda\bm{\omega}/2|^2}{2}\right)
=z\exp\left(-\frac{q^2+\lambda^2/4\pm q\lambda\xi}{2}\right).
\end{equation*}
Thus $\mathcal{N}\mathcal{N}_*=(z/\theta_0)^2e^{-q^2-\lambda^2/4}$.
The incoming and outgoing pairs are $\bm{c}\pm s\bm{\omega}/2$ and $\bm{c}\pm s\bm{\sigma}/2$, so the two angular integrals in the numerator of \eqref{eq:radial_projection} separate and coincide.
Hence,
\begin{equation*}
\begin{aligned}
&\int_{\mathbb{R}^3}\int_{\mathbb{S}^2}\int_{\mathbb{S}^2}
\mathcal{N}\mathcal{N}_*\mathcal{A}(\mathcal{N})
\,\dif\bm{\sigma}\,\dif\bm{\omega}\,\dif\bm{c}\\
&\quad=\left(\frac{z}{\theta_0}\right)^2T^{3/2}e^{-\lambda^2/4}
\int_{\mathbb{R}^3}e^{-|\bm{y}|^2}
\left[
\int_{\mathbb{S}^2}
\frac{\dif\bm{\omega}}{
\left(1-\tau z e^{-\frac{|\bm{y}+\lambda\bm{\omega}/2|^2}{2}}\right)
\left(1-\tau z e^{-\frac{|\bm{y}-\lambda\bm{\omega}/2|^2}{2}}\right)}
\right]^2\dif\bm{y}\\
&\quad=\left(\frac{z}{\theta_0}\right)^2T^{3/2}e^{-\lambda^2/4}(4\pi)^2
\int_{\mathbb{R}^3}e^{-|\bm{y}|^2}
\left[
\frac12\int_{-1}^{1}
\frac{\dif\xi}{
\left(1-\tau z e^{-\frac{q^2+\lambda^2/4+q\lambda\xi}{2}}\right)
\left(1-\tau z e^{-\frac{q^2+\lambda^2/4-q\lambda\xi}{2}}\right)}
\right]^2\dif\bm{y}.
\end{aligned}
\end{equation*}
For the denominator, the same change of variables leaves a Gaussian integral of value $\pi^{3/2}$ and two spherical integrals, each of value $4\pi$.
Hence,
\begin{equation*}
\int_{\mathbb{R}^3}\int_{\mathbb{S}^2}\int_{\mathbb{S}^2}
\mathcal{N}\mathcal{N}_*
\,\dif\bm{\sigma}\,\dif\bm{\omega}\,\dif\bm{c}
=\left(\frac{z}{\theta_0}\right)^2T^{3/2}e^{-\lambda^2/4}
\pi^{3/2}(4\pi)^2.
\end{equation*}
Eliminating the common factor $(z/\theta_0)^2T^{3/2}e^{-\lambda^2/4}(4\pi)^2$ and integrating radially in $\bm{y}$ gives
\begin{equation}
\label{eq:radial_angular_average}
\begin{aligned}
\widetilde{\mathcal{A}}_{\mathcal{N}}(s)
&=\frac{1}{\pi^{3/2}}\int_{\mathbb{R}^3}e^{-|\bm{y}|^2}
\left[
\frac12\int_{-1}^{1}
\frac{\dif\xi}{
\left(1-\tau z e^{-\frac{q^2+\lambda^2/4+q\lambda\xi}{2}}\right)
\left(1-\tau z e^{-\frac{q^2+\lambda^2/4-q\lambda\xi}{2}}\right)}
\right]^2\dif\bm{y}\\
&=\frac{1}{\sqrt{\pi}}\int_0^\infty q^2e^{-q^2}
\left[
\int_{-1}^{1}
\frac{\dif\xi}{
\left(1-\tau z e^{-\frac{q^2+\lambda^2/4+q\lambda\xi}{2}}\right)
\left(1-\tau z e^{-\frac{q^2+\lambda^2/4-q\lambda\xi}{2}}\right)}
\right]^2\dif q
=\mathcal{A}_\tau(z,\lambda).
\end{aligned}
\end{equation}
For fixed $z$, the radial factor is independent of $\bm{u}$ and $\theta_0$, and its dependence on $T$ is only through $\lambda=s/\sqrt{T}$.
\qed
\end{proof}

\begin{remark}
We estimate the computational cost of tabulating $\widetilde{\mathcal{A}}_{\mathcal{N}}$ using \eqref{eq:easy_version_aux}.
For prescribed $z$ and $T$, tabulate \eqref{eq:easy_version_aux} at $J=O(N_v)$ radial nodes $\{s_k\}_{k=1}^J$.
With $P=O(N_v)$ quadrature nodes for each of the $q$ and $\xi$ integrals, this requires $O(JP^2)=O(N_v^3)$ operations.
This preprocessing cost is of lower asymptotic order than the target $O(N_v^4\log N_v)$ cost of collision evaluation.
\end{remark}
\subsection{Conservation and the Linearized H-Theorem}
Replacing $\mathcal{A}(\mathcal{N})$ in \eqref{eq:linearized_collision} by its radial projection gives the effective linear operator
\begin{equation}
\label{eq:augmented_L}
\widetilde{\mathcal{L}}_q(r)(\bm{v})
=\int_{\mathbb{R}^3}\int_{\mathbb{S}^2}
B(s)\widetilde{\mathcal{A}}_{\mathcal{N}}(s)
\mathcal{N}\mathcal{N}_*
(r'+r_*'-r-r_*)\,\dif\bm{\sigma}\,\dif\bm{v}_*,
\qquad s=|\bm{v}-\bm{v}_*|.
\end{equation}
Below, we are going to derive its weak form, which will further yield the conservation properties and the linearized H-theorem.
Here we assume that all functions are real-valued and all integrals below converge absolutely.

\begin{theorem}[Weak form]
\label{Thm:Weak}
Let $\Delta r=r'+r_*'-r-r_*$ and define $\Delta k$ analogously for a test function $k$.
Then
\begin{equation}
\label{eq:weak}
\int_{\mathbb{R}^3}k\,\widetilde{\mathcal{L}}_q(r)\,\dif\bm{v}
=-\frac14\int_{\mathbb{R}^3}\int_{\mathbb{R}^3}\int_{\mathbb{S}^2}
B(s)\widetilde{\mathcal{A}}_{\mathcal{N}}(s)
\mathcal{N}\mathcal{N}_*\,\Delta k\,\Delta r
\,\dif\bm{\sigma}\,\dif\bm{v}_*\,\dif\bm{v}.
\end{equation}
\end{theorem}

\begin{proof}
To obtain $\Delta k$, we symmetrize the test factor over the four collision velocities. 
Recall the collision coordinates $(\bm{c},s,\bm{\omega},\bm{\sigma})$ from Theorem~\ref{thm:radial_projection}, with $\bm{g}=s\bm{\omega}$ and
\begin{equation*}
\dif\bm{v}\,\dif\bm{v}_*=\dif\bm{c}\,\dif\bm{g},
\qquad
\dif\bm{g}=s^2\,\dif s\,\dif\bm{\omega}.
\end{equation*}
The inverse relations are $\bm{v}=\bm{c}+s\bm{\omega}/2$ and $\bm{v}_*=\bm{c}-s\bm{\omega}/2$, with $\bm{\sigma}$ unchanged.
Using $\mathcal{N}\mathcal{N}_*=\mathcal{N}^{c}(\bm{c})e^{-s^2/(4T)}$ from \eqref{eq:Maxwellian_Separate}, we obtain
\begin{equation*}
\begin{aligned}
&\int_{\mathbb{R}^3}k\,\widetilde{\mathcal{L}}_q(r)\,\dif\bm{v}=\int_0^\infty\int_{\mathbb{R}^3}
s^2B(s)\widetilde{\mathcal{A}}_{\mathcal{N}}(s)
\mathcal{N}^{c}(\bm{c})e^{-s^2/(4T)}
\left[
\int_{\mathbb{S}^2}\int_{\mathbb{S}^2}
k\left(\bm{c}+\frac{s\bm{\omega}}2\right)\Delta r
\,\dif\bm{\sigma}\,\dif\bm{\omega}
\right]\dif\bm{c}\,\dif s.
\end{aligned}
\end{equation*}
The weight outside the angular integral depends only on $\bm{c}$ and $s$.

The substitution $(\bm{\omega},\bm{\sigma})\mapsto(-\bm{\omega},-\bm{\sigma})$ gives
\begin{equation*}
\int_{\mathbb{S}^2}\int_{\mathbb{S}^2}
k\left(\bm{c}+\frac{s\bm{\omega}}2\right)\Delta r
\,\dif\bm{\sigma}\,\dif\bm{\omega}
=
\int_{\mathbb{S}^2}\int_{\mathbb{S}^2}
k\left(\bm{c}-\frac{s\bm{\omega}}2\right)\Delta r
\,\dif\bm{\sigma}\,\dif\bm{\omega}.
\end{equation*}
Next, $(\bm{\omega},\bm{\sigma})\mapsto(\bm{\sigma},\bm{\omega})$ gives, for either sign,
\begin{equation*}
\int_{\mathbb{S}^2}\int_{\mathbb{S}^2}
k\left(\bm{c}\pm\frac{s\bm{\omega}}2\right)\Delta r
\,\dif\bm{\sigma}\,\dif\bm{\omega}
=
-\int_{\mathbb{S}^2}\int_{\mathbb{S}^2}
k\left(\bm{c}\pm\frac{s\bm{\sigma}}2\right)\Delta r
\,\dif\bm{\sigma}\,\dif\bm{\omega}.
\end{equation*}
Thus the original integral has four equivalent representations with test factors $k$, $k_*$, $-k'$, and $-k_*'$.
Averaging them and returning to $(\bm{v},\bm{v}_*,\bm{\sigma})$ gives
\begin{equation*}
\int_{\mathbb{R}^3}k\,\widetilde{\mathcal{L}}_q(r)\,\dif\bm{v}
=\frac14\int_{\mathbb{R}^3}\int_{\mathbb{R}^3}\int_{\mathbb{S}^2}
B(s)\widetilde{\mathcal{A}}_{\mathcal{N}}(s)
\mathcal{N}\mathcal{N}_*(k+k_*-k'-k_*')\Delta r
\,\dif\bm{\sigma}\,\dif\bm{v}_*\,\dif\bm{v}.
\end{equation*}
Since $k+k_*-k'-k_*'=-\Delta k$, this is \eqref{eq:weak}.
\qed
\end{proof}

Taking $k$ to be a collision invariant gives $\Delta k=0$ and hence conservation of mass, momentum, and energy:
\begin{equation}
\label{eq:Conservation}
\int_{\mathbb{R}^3}
\begin{pmatrix}1\\ \bm{v}\\ |\bm{v}|^2\end{pmatrix}
\widetilde{\mathcal{L}}_q(r)\,\dif\bm{v}=\bm{0}.
\end{equation}
Thus the approximate kinetic equation
\begin{equation}
\label{eq:approx_kinetic}
\frac{\partial f}{\partial t}+\bm{v}\cdot\nabla_{\bm{x}}f
=\frac{1}{\Kn}\widetilde{\mathcal{L}}_q
\left(\frac{f-\mathcal{M}_q}{\mathcal{M}_q(1+\tau\theta_0\mathcal{M}_q)}\right)
\end{equation}
has the same moment balances as \eqref{eq:QBE}:
\begin{equation}
\label{eq:Quantum_Euler}
\begin{aligned}
&\frac{\partial\rho}{\partial t}
+\nabla_{\bm{x}}\cdot(\rho\bm{u})=0,\\
&\frac{\partial(\rho\bm{u})}{\partial t}
+\nabla_{\bm{x}}\cdot(\rho\bm{u}\otimes\bm{u}+\mathbb{P})=0,\\
&\frac{\partial}{\partial t}
\left(\frac32p+\frac12\rho|\bm{u}|^2\right)
+\nabla_{\bm{x}}\cdot
\left[\left(\frac32p+\frac12\rho|\bm{u}|^2\right)\bm{u}
+\mathbb{P}\bm{u}+\bm{q}\right]=0,
\end{aligned}
\end{equation}
with the moments defined in \eqref{eq:momentums}.
The equilibrium closure $\mathbb{P}=pI$ and $\bm{q}=\bm{0}$ gives the quantum Euler system \citep{filbet2010numerical}.

\begin{theorem}[Linearized H-theorem]
\label{thm:linearized_H}
At fixed time and position, let $\mathcal{M}_q$ be an admissible equilibrium.
For the VHS kernel \eqref{eq:B_VHS} with $C_B>0$ and any admissible
$r\in L^2(\mathbb{R}^3;\mathcal{M}_q(1+\tau\theta_0\mathcal{M}_q)\,\dif\bm{v})$, $\widetilde{\mathcal{L}}_q$ satisfies \textit{linearized Boltzmann's H-Theorem}
\begin{equation*}
\int_{\mathbb{R}^3}r\,\widetilde{\mathcal{L}}_q(r)\,\dif\bm{v}\le0.
\end{equation*}
Set $r=(f-\mathcal{M}_q)/[\mathcal{M}_q(1+\tau\theta_0\mathcal{M}_q)]$.
If $\mathcal{M}_q$ has the same mass, momentum, and energy as $f$, then
\begin{equation}
\label{eq:linearized_H_equality}
\int_{\mathbb{R}^3}r\,\widetilde{\mathcal{L}}_q(r)\,\dif\bm{v}=0
\quad\Longleftrightarrow\quad
\widetilde{\mathcal{L}}_q(r)=0
\quad\Longleftrightarrow\quad
f=\mathcal{M}_q.
\end{equation}
\end{theorem}

\begin{proof}
Taking $k=r$ in \eqref{eq:weak} gives
\begin{equation}
\label{eq:L_dissipation}
\int_{\mathbb{R}^3}r\,\widetilde{\mathcal{L}}_q(r)\,\dif\bm{v}
=-\frac14\int_{\mathbb{R}^3}\int_{\mathbb{R}^3}\int_{\mathbb{S}^2}
B(s)\widetilde{\mathcal{A}}_{\mathcal{N}}(s)
\mathcal{N}\mathcal{N}_*|\Delta r|^2
\,\dif\bm{\sigma}\,\dif\bm{v}_*\,\dif\bm{v}\le0.
\end{equation}
The collision weight is positive almost everywhere, so equality holds precisely when $\Delta r=0$ almost everywhere.
This is equivalent to $\widetilde{\mathcal{L}}_q(r)=0$ by \eqref{eq:augmented_L} and \eqref{eq:L_dissipation}.
The characterization of collision invariants \citep{ArkerydCercignani1990} then gives
$r\in\operatorname{span}\{1,v_1,v_2,v_3,|\bm{v}|^2\}$.
The moment-matching condition gives
\begin{equation*}
\int_{\mathbb{R}^3}
(f-\mathcal{M}_q)
\begin{pmatrix}1\\ \bm{v}\\ |\bm{v}|^2\end{pmatrix}
\,\dif\bm{v}=\bm{0}.
\end{equation*}
Since $r$ is a linear combination of these collision invariants,
\begin{equation*}
0=\int_{\mathbb{R}^3}(f-\mathcal{M}_q)r\,\dif\bm{v}
=\int_{\mathbb{R}^3}\mathcal{M}_q(1+\tau\theta_0\mathcal{M}_q)r^2\,\dif\bm{v},
\end{equation*}
where the second equality uses the stated choice of $r$.
The weight $\mathcal{M}_q(1+\tau\theta_0\mathcal{M}_q)$ is strictly positive, so $r=0$ almost everywhere and $f=\mathcal{M}_q$.
Conversely, $f=\mathcal{M}_q$ gives $r=0$, proving \eqref{eq:linearized_H_equality}.
\qed
\end{proof}

% \begin{remark}
% The entropy variable in \eqref{eq:H_Thm} has derivative
% $1/[f(1+\tau\theta_0f)]$.
% Thus $r=(f-\mathcal{M}_q)/[\mathcal{M}_q(1+\tau\theta_0\mathcal{M}_q)]$
% is its first-order change about $\mathcal{M}_q$.
% Its equilibrium value
% $\log\mathcal{N}=\log(z/\theta_0)-|\bm{v}-\bm{u}|^2/(2T)$
% is a collision invariant and contributes nothing to the collision entropy production.
% The corresponding convex quantum H-functional has quadratic part
% \begin{equation*}
% \frac12\int_{\mathbb{R}^3}
% \frac{(f-\mathcal{M}_q)^2}{\mathcal{M}_q(1+\tau\theta_0\mathcal{M}_q)}
% \,\dif\bm{v}
% =\frac12\int_{\mathbb{R}^3}
% \mathcal{M}_q(1+\tau\theta_0\mathcal{M}_q)r^2\,\dif\bm{v}.
% \end{equation*}
% Accordingly, \eqref{eq:L_dissipation} is the quadratic counterpart of \eqref{eq:H_Thm}, and \eqref{eq:linearized_H_equality} is the linear analogue of the equilibrium characterization in \citet[Eq.~(2.7)]{filbet2010numerical}.
% Moment matching removes the nonzero collision-invariant modes, leaving $f=\mathcal{M}_q$ as the only case of equality.
% This does not assert the full nonlinear entropy inequality for the effective operator.
% \end{remark}

\begin{theorem}[Kernel]
\label{thm:effective_kernel}
For the VHS kernel \eqref{eq:B_VHS} with $C_B>0$, the kernel of $\widetilde{\mathcal{L}}_q$, acting on admissible
$r\in L^2(\mathbb{R}^3;\mathcal{M}_q(1+\tau\theta_0\mathcal{M}_q)\,\dif\bm{v})$, is
\begin{equation*}
\ker\widetilde{\mathcal{L}}_q
=\operatorname{span}\{1,v_1,v_2,v_3,|\bm{v}|^2\}.
\end{equation*}
\end{theorem}

\begin{proof}
The proof of Theorem~\ref{thm:linearized_H} places every null mode in the stated span.
Conversely, each function in this span has $\Delta r=0$ and is therefore annihilated by \eqref{eq:augmented_L}.
\qed
\end{proof}

\begin{remark}
In terms of the physical perturbation $f-\mathcal{M}_q=\mathcal{M}_q(1+\tau\theta_0\mathcal{M}_q)r$, the null perturbations form the space
\begin{equation*}
\mathcal{M}_q(1+\tau\theta_0\mathcal{M}_q)
\operatorname{span}\{1,v_1,v_2,v_3,|\bm{v}|^2\}.
\end{equation*}
This is the tangent space of the quantum equilibrium manifold at $\mathcal{M}_q$, obtained by varying $z$, $\bm{u}$, and $T$ in \eqref{eq:local_equilibrium_QBE}.
Thus the approximation retains the equilibrium tangent modes of the exact linearization.
Equality of the kernels does not, however, imply equality of the nonzero spectra, relaxation rates, or transport coefficients.
\end{remark}
\subsection{Fast Spectral Realization}
The radial factor in \eqref{eq:augmented_L} allows us to apply the fast algorithm of \citet{yin2026fast}.
With $\eta=\mathcal{N}r$, write
$\widetilde{\mathcal{L}}_q(r)=\widetilde{\mathcal{L}}^+(r)-\widetilde{\mathcal{L}}^-(r)$,
where the gain and loss terms are
\begin{equation}
\label{eq:effective_linearized_form}
\begin{aligned}
\widetilde{\mathcal{L}}^+(r)(\bm{v})
&=\int_{\mathbb{R}^3}\int_{\mathbb{S}^2}
B(s)\widetilde{\mathcal{A}}_{\mathcal{N}}(s)
\bigl(\mathcal{N}_*'\eta'+\mathcal{N}'\eta_*'\bigr)
\,\dif\bm{\sigma}\,\dif\bm{v}_*,\\
\widetilde{\mathcal{L}}^-(r)(\bm{v})
&=\int_{\mathbb{R}^3}\int_{\mathbb{S}^2}
B(s)\widetilde{\mathcal{A}}_{\mathcal{N}}(s)
\bigl(\mathcal{N}_*\eta+\mathcal{N}\eta_*\bigr)
\,\dif\bm{\sigma}\,\dif\bm{v}_*,
\qquad s=|\bm{v}-\bm{v}_*|.
\end{aligned}
\end{equation}
We retain the domain $\mathcal{D}_L$, cutoff $R$, grid, and Fourier convention of Section~\ref{sec:FSM_Direct}.
With sums over the retained Fourier modes, the loss coefficients have the form
\begin{equation}
\label{eq:effective_loss_modes}
\begin{aligned}
\widehat{\widetilde{\mathcal{L}}}^{-}_{\bm{k}}
&\approx16\pi^2\sum_{\bm{\ell}}
\bigl(\widehat{\mathcal{N}}_{\bm{k}-\bm{\ell}}\widehat{\eta}_{\bm{\ell}}
+\widehat{\eta}_{\bm{k}-\bm{\ell}}\widehat{\mathcal{N}}_{\bm{\ell}}\bigr)\times\int_0^R s^2B(s)\widetilde{\mathcal{A}}_{\mathcal{N}}(s)
\operatorname{sinc}\!\left(\frac{\pi s|\bm{\ell}|}{L}\right)\,\dif s,
\end{aligned}
\end{equation}
where $\operatorname{sinc}(a)=\sin(a)/a$ and $\operatorname{sinc}(0)=1$.
% With the normalization \eqref{eq:Fourier_expansion}, the coefficient convolution carries no additional factor $(2L)^{-3}$; $16\pi^2$ comes from the two spherical integrations.
The radial integral depends only on $\bm{\ell}$ and can be precomputed, leaving two convolutions.

The post-collision velocities couple the angular dependence in the gain term.
Following \citet{yin2026fast}, we use $r=\eta/\mathcal{N}$ and the Gaussian factorization \eqref{eq:Maxwellian_Separate} to obtain
\begin{equation}
\label{eq:effective_gain_modes}
\widehat{\widetilde{\mathcal{L}}}^{+}_{\bm{k}}
\approx\frac{4\pi^2}{L^3}\int_0^R
s^2B(s)\widetilde{\mathcal{A}}_{\mathcal{N}}(s)
\operatorname{sinc}\!\left(\frac{\pi s|\bm{k}|}{2L}\right)
\left[\int_{\mathcal{D}_L}\mathcal{N}^{c}(\bm{v})\psi(s,\bm{v})
e^{-i\pi\bm{k}\cdot\bm{v}/L}\,\dif\bm{v}\right]\,\dif s,
\end{equation}
where
\begin{equation}
\label{eq:effective_gain_convolution}
\psi(s,\bm{v})
=\frac{1}{(2L)^3}\int_{\mathcal{D}_L}
r(\bm{v}-\bm{w})e^{-|\bm{w}|^2/T}
\left[\sum_{\bm{\ell}}
\operatorname{sinc}\!\left(\frac{\pi s|\bm{\ell}|}{2L}\right)
e^{i\pi\bm{\ell}\cdot\bm{w}/L}\right]\,\dif\bm{w}.
\end{equation}
Here the sum uses the retained Fourier modes and $r$ is extended periodically on $\mathcal{D}_L$.
For each fixed $s$, $\psi(s,\cdot)$ is a convolution with a Gaussian-weighted kernel.
The bracketed sum and this convolution are evaluated by FFTs; multiplication by $\mathcal{N}^{c}$ and another FFT give the coefficients in \eqref{eq:effective_gain_modes}. Evaluating $r=(f-\mathcal{M}_q)/[\mathcal{M}_q(1+\tau\theta_0\mathcal{M}_q)]$ in the tails can amplify round-off errors. We therefore set $r=0$ where $\mathcal{M}_q$ falls below a prescribed cutoff before applying the FFT convolution.
We refer to \citet{yin2026fast} for the derivation and numerical details.

With $J=O(N_v)$ radial nodes, FFT evaluation of these convolutions costs $O(JN_v^3\log N_v)=O(N_v^4\log N_v)$ for the gain term; the loss term costs $O(N_v^3\log N_v)$ after precomputation.
The tabulated factor $\widetilde{\mathcal{A}}_{\mathcal{N}}$ changes only the radial weights.
Analytical angular integration removes the factor $M$ in Section~\ref{sec:FSM_Direct}.

\section{Application to the Steady-State Quantum Boltzmann Equation}
\label{sec:newton}
In this section, we use the effective linear operator \eqref{eq:augmented_L} in a modified Newton method for the steady-state quantum Boltzmann equation. Since the radial projection introduces an approximation error relative to the exact linearization \eqref{eq:linearized_collision}, the resulting operator is used only to approximate the Fr\'echet derivative of $\mathcal{Q}_q$ in the Newton correction equation. The nonlinear residual is still evaluated with the full quantum collision operator $\mathcal{Q}_q$.
\subsection{Steady-State Formulation and Approximate Jacobian}
We consider the one-dimensional steady-state quantum Boltzmann equation on $x_1\in(x_L,x_R)$, with $\bm{v}\in\mathbb{R}^3$:
\begin{equation}
    \label{eq:1D_Steady_QBE}
    v_1\pdfFrac{f}{x_1}=\frac{1}{\Kn}\mathcal{Q}_{q}(f).
\end{equation}
Writing $T_L,T_R$ for the thermodynamic wall temperatures and $\bm{u}_L,\bm{u}_R$ for the prescribed tangential velocities, the incoming distributions are
\begin{equation}
    \label{eq:Boundary_Cond_for_1D_QBE}
    \begin{aligned}
    f(x_L,\bm{v};\tau)
    &=\frac{1}{\theta_0}
      \frac{z_L}{\exp\left(\frac{|\bm{v}-\bm{u}_L|^2}{2T_L}\right)-\tau z_L},
      \qquad v_1>0,\\
    f(x_R,\bm{v};\tau)
    &=\frac{1}{\theta_0}
      \frac{z_R}{\exp\left(\frac{|\bm{v}-\bm{u}_R|^2}{2T_R}\right)-\tau z_R},
      \qquad v_1<0.
    \end{aligned}
\end{equation}
The wall fugacities $z_L,z_R$ are chosen to enforce zero normal mass flux at each boundary:
\begin{equation}
    \label{eq:zero_flux_BC}
    \int_{\mathbb{R}^3}v_1f(x_L,\bm{v};\tau)\dif\bm{v}
    =\int_{\mathbb{R}^3}v_1f(x_R,\bm{v};\tau)\dif\bm{v}=0.
\end{equation}

Define the steady-state residual by
\begin{equation}
    \label{eq:Residual_Functional}
    \mathcal{R}(f)\coloneqq v_1\pdfFrac{f}{x_1}
    -\frac{1}{\Kn}\mathcal{Q}_q(f),
\end{equation}
Newton's method determines the correction to the current iterate by linearizing the steady-state residual $\mathcal{R}(f)$. This linearization requires the Fr\'echet derivative of the nonlinear collision operator $\mathcal{Q}_q$, which describes its first-order response to perturbations of $f$.
For a perturbation $g$, let $\mathrm{D}_f\mathcal{Q}_q[f](g)$ denote the Fr\'echet derivative of the collision operator with respect to the physical distribution $f$, in the direction $g$.
The exact derivative of the residual is
\begin{equation}
    \label{eq:exact_Frechet_residual}
    \mathrm{D}_f\mathcal{R}[f](g)
    =\left.\frac{\dif}{\dif\varepsilon}
      \mathcal{R}(f+\varepsilon g)\right|_{\varepsilon=0}
    =v_1\pdfFrac{g}{x_1}
     -\frac{1}{\Kn}\mathrm{D}_f\mathcal{Q}_q[f](g).
\end{equation}

Applying $\mathrm{D}_f\mathcal{Q}_q[f]$ directly is expensive.
We first consider the special case $\mathrm{D}_f\mathcal{Q}_q[\mathcal{M}_q]$, where $\mathcal{M}_q$ is the local quantum equilibrium with the same mass, momentum, and energy as $f$.
We then apply the radial approximation of Section~\ref{sec:linear-approx} to this equilibrium linearization.
Set $\mathcal{N}=\mathcal{M}_q/(1+\tau\theta_0\mathcal{M}_q)$.
By the chain rule and \eqref{eq:linearized_collision}, the physical linearization and its effective approximation satisfy
\begin{equation}
    \label{eq:modified_Frechet_operator}
    \begin{aligned}
    \mathrm{D}_f\mathcal{Q}_q[\mathcal{M}_q](g)
    &=\mathcal{L}_q
    \left(\frac{g}{\mathcal{M}_q(1+\tau\theta_0\mathcal{M}_q)}\right),\quad\widetilde{\mathcal{L}}_q[g]
    \coloneqq\widetilde{\mathcal{L}}_q
    \left(\frac{g}{\mathcal{M}_q(1+\tau\theta_0\mathcal{M}_q)}\right)
    \approx\mathrm{D}_f\mathcal{Q}_q[\mathcal{M}_q](g).
    \end{aligned}
\end{equation}
Here $\widetilde{\mathcal{L}}_q$ on the right-hand side is constructed with the reference equilibrium $\mathcal{N}$.
Square brackets in $\widetilde{\mathcal{L}}_q[g]$ indicate the physical perturbation as input, whereas parentheses in $\widetilde{\mathcal{L}}_q(r)$ indicate the weighted perturbation $r=g/\bigl(\mathcal{M}_q(1+\tau\theta_0\mathcal{M}_q)\bigr)$.
For a given outer Newton iterate, the local equilibrium and collision coefficients are held fixed throughout the inner fixed-point iteration.

\subsection{Modified Newton Algorithm}
The Newton correction associated with the full step $f^{(k)}-g^{(k)}$ satisfies
\begin{equation}
    \label{eq:Newton}
    v_1\partial_{x_1}(g^{(k)}-f^{(k)})
    =\frac{1}{\Kn}\left(
      \mathrm{D}_f\mathcal{Q}_q[f^{(k)}](g^{(k)})
      -\mathcal{Q}_q(f^{(k)})\right).
\end{equation}
For each outer Newton iterate $f^{(k)}$, let $\mathcal{M}_q^{(k)}$ denote the local equilibrium determined by its moments, with transform $\mathcal{N}^{(k)}$ and corresponding effective linear operator $\widetilde{\mathcal{L}}_q^{(k)}$.
 Replacing $\mathrm{D}_f\mathcal{Q}_q[f^{(k)}]$ by the effective operator in \eqref{eq:modified_Frechet_operator} gives
\begin{equation}
    \label{eq:Modified_Newton}
    v_1\partial_{x_1}(g^{(k)}-f^{(k)})
    =\frac{1}{\Kn}\left(
      \widetilde{\mathcal{L}}_q^{(k)}[g^{(k)}]
      -\mathcal{Q}_q(f^{(k)})\right).
\end{equation}
Following \citet{yin2026fast}, we solve \eqref{eq:Modified_Newton} by the inner fixed-point iteration, initialized with $g^{(k,0)}=0$:
\begin{equation}
    \label{eq:Newton_Iter}
    \begin{aligned}
    v_1\partial_{x_1}(g^{(k,l+1)}-f^{(k)})
     +\frac{\nu(x_1)}{\Kn}g^{(k,l+1)}=\frac{1}{\Kn}\left(
      \widetilde{\mathcal{L}}_q^{(k)}[g^{(k,l)}]
      +\nu(x_1)g^{(k,l)}-\mathcal{Q}_q(f^{(k)})\right).
    \end{aligned}
\end{equation}
Here $l$ indexes the inner iterations, and $\nu(x_1)>0$ is a stabilization frequency fixed for each outer Newton iterate.
A practical choice, used by \citet[Supplementary Material, Section~SM2]{yin2026fast} for the classical Boltzmann equation, is $\nu(x_1)=\rho^{(k)}(x_1)$, where $\rho^{(k)}$ is the local density of $f^{(k)}$.
The added term treats a damping contribution implicitly and cancels at a fixed point, leaving \eqref{eq:Modified_Newton} unchanged.

We discretize \eqref{eq:Newton_Iter} by a finite volume method with first-order upwind fluxes.
Divide $[x_L,x_R]$ into $N_x$ uniform cells
$C_i=[x_{i-\frac12},x_{i+\frac12}]$ of width
$\Delta x=(x_R-x_L)/N_x$, with centres
$x_i=x_L+(i-\frac12)\Delta x$.
The index $i=1,\ldots,N_x$ labels spatial cells, while $j$ labels the nodes
$\bm{v}_j=(v_{1,j},v_{2,j},v_{3,j})$ of the three-dimensional velocity grid.
The discrete unknowns represent cell averages,
\begin{equation}
    \label{eq:fv_cell_average}
    f_{i,j}^{(k)}
    =\frac{1}{\Delta x}\int_{C_i}
      f^{(k)}(x_1,\bm{v}_j)\,\dif x_1,
\end{equation}
with $g_{i,j}^{(k,l)}$ defined analogously.
Collision operators are evaluated separately in each cell using its full velocity distribution; the subscript $(i,j)$ denotes their value at $\bm{v}_j$ in cell $i$.
Writing $\nu_i=\nu(x_i)$, we obtain, for $v_{1,j}>0$,
\begin{equation}
    \label{eq:upwind}
    \begin{aligned}
    &v_{1,j}
      \frac{g_{i,j}^{(k,l+1)}-g_{i-1,j}^{(k,l+1)}}{\Delta x}
     -v_{1,j}
      \frac{f_{i,j}^{(k)}-f_{i-1,j}^{(k)}}{\Delta x}
     +\frac{\nu_i}{\Kn}g_{i,j}^{(k,l+1)}=\frac{1}{\Kn}\left[
      \left(\widetilde{\mathcal{L}}_q^{(k)}[g^{(k,l)}]\right)_{i,j}
      +\nu_i g_{i,j}^{(k,l)}
      -\mathcal{Q}_q(f^{(k)})_{i,j}\right].
    \end{aligned}
\end{equation}
The sweep proceeds from $i=1$ to $N_x$.
The indices $i=0$ and $i=N_x+1$ are reserved for the left and right boundary traces.
At the left wall, the incoming correction is set by
\begin{equation}
    \label{eq:Left_BC}
    f_{0,j}^{(k)}-g_{0,j}^{(k,l+1)}
    =\frac{1}{\theta_0}
      \frac{z_L}{\exp\!\left(\frac{|\bm{v}_j-\bm{u}_L|^2}{2T_L}\right)-\tau z_L},
      \qquad v_{1,j}>0,
\end{equation}
where $z_L$ is the wall fugacity used for the current sweep.
For $v_{1,j}<0$, use forward differences and sweep $i=N_x,\ldots,1$, with the right-wall analogue of \eqref{eq:Left_BC}.
When $v_{1,j}=0$, the transport term vanishes and the update is local in space.
Once all interior values are available, we recompute the wall fugacities and incoming traces from the candidate $f^{(k)}-g^{(k,l+1)}$ to satisfy the discrete zero-flux conditions before evaluating the residual.
Each sweep therefore solves linear transport equations with fixed collision coefficients and boundary data, while the inner iteration also includes the nonlinear wall-fugacity update.

We monitor the discrete residual of \eqref{eq:Modified_Newton}.
For $v_{1,j}>0$, its cellwise expression is
\begin{equation}
    \label{eq:inner_residual}
    \begin{aligned}
    e_{i,j}^{(k,l)}
    &=\frac{v_{1,j}}{\Delta x}
      \left(g_{i,j}^{(k,l)}-g_{i-1,j}^{(k,l)}
      -f_{i,j}^{(k)}+f_{i-1,j}^{(k)}\right)-\frac{1}{\Kn}\left[
      \left(\widetilde{\mathcal{L}}_q^{(k)}[g^{(k,l)}]\right)_{i,j}
      -\mathcal{Q}_q(f^{(k)})_{i,j}\right],\\
    \mathcal{E}_{k,l}
    &=\left(\Delta x(\Delta v)^3
      \sum_{i=1}^{N_x}\sum_j
      \left|e_{i,j}^{(k,l)}\right|^2\right)^{1/2}.
    \end{aligned}
\end{equation}
For $v_{1,j}<0$, use forward differences in the transport term; for $v_{1,j}=0$, this term vanishes.
The norm includes all velocity nodes and interior cells, with boundary traces entering through the difference quotients.
For an absolute tolerance $\varepsilon_{\mathrm{abs}}>0$ and a relative tolerance $0<\varepsilon_{\mathrm{rel}}<1$, we stop when $\mathcal{E}_{k,l+1}\le\max\left\{\varepsilon_{\mathrm{abs}},\varepsilon_{\mathrm{rel}}\mathcal{E}_{k,0}\right\},$ and set $g^{(k)}=g^{(k,l+1)}$.
\subsection{Boundary Treatment and Discrete Constraints}
After each inner sweep, we update $z_L$ and $z_R$ from the outgoing traces of the candidate $f^{(k)}-g^{(k,l+1)}$.
With the uniform-grid velocity weights $w_j^v=(\Delta v)^3$, the discrete zero-flux conditions read
\begin{equation}
    \label{eq:BCs}
  \begin{aligned}
    &\sum_{v_{1,j}>0} w_j^v v_{1,j}
      \frac{1}{\theta_0}
      \frac{z_L}{\exp\!\left(\frac{|\bm v_j-\bm u_L|^2}{2T_L}\right)-\tau z_L}
      +\sum_{v_{1,j}<0} w_j^v v_{1,j}
      \left(f_{1,j}^{(k)}-g_{1,j}^{(k,l+1)}\right)=0,\\
    &\sum_{v_{1,j}>0} w_j^v v_{1,j}
      \left(f_{N_x,j}^{(k)}-g_{N_x,j}^{(k,l+1)}\right)
      +\sum_{v_{1,j}<0} w_j^v v_{1,j}
      \frac{1}{\theta_0}
      \frac{z_R}{\exp\!\left(\frac{|\bm v_j-\bm u_R|^2}{2T_R}\right)-\tau z_R}=0.
  \end{aligned}
\end{equation}
For Couette and Fourier flow with prescribed wall velocities $\bm{u}_L,\bm{u}_R$ and kinetic temperatures $\theta_L,\theta_R$, the thermodynamic wall temperatures satisfy $T_s=\theta_s\,\operatorname{Li}_{\frac{3}{2}}(\tau z_s)/\operatorname{Li}_{\frac{5}{2}}(\tau z_s)$, $s=L,R$, by \eqref{eq:fugality_and_temperature}.
Substituting these relations into \eqref{eq:BCs} gives a scalar nonlinear equation for each wall fugacity.

The continuous collision operators satisfy \eqref{eq:conservation_Qq} and \eqref{eq:Conservation}, but their spectral discretizations need not preserve all five moments exactly.
Fix a spatial cell and an outer iterate, and suppress these indices.
Set $\bm{\phi}_j=(1,v_{1,j},v_{2,j},v_{3,j},|\bm{v}_j|^2)^{\mathsf T}$ and $\chi_j=\mathcal{M}_{q,j}(1+\tau\theta_0\mathcal{M}_{q,j})$.
For either collision output $q_j$, constrained least squares with weights $w_j^v/\chi_j$ gives
\begin{equation}
    \label{eq:collision_moment_correction}
    \left(\sum_j w_j^v\chi_j
      \bm{\phi}_j\bm{\phi}_j^{\mathsf T}\right)\bm{a}
    =\sum_j w_j^v\bm{\phi}_j q_j,
    \qquad
    q_j^{\mathrm c}=q_j-\chi_j\bm{\phi}_j^{\mathsf T}\bm{a}.
\end{equation}
This subtracts a component in the equilibrium tangent space and gives $\sum_j w_j^v\bm{\phi}_j q_j^{\mathrm c}=\bm{0}$, up to roundoff.
In the updates and residuals, we apply this correction to the outputs of both $\mathcal{Q}_q$ and $\widetilde{\mathcal{L}}_q$, enforcing discrete conservation of mass, momentum, and energy for both operators.

For the classical Boltzmann equation, \citet{filbet2015steady} preserve the Maxwellian equilibrium through a micro--macro decomposition before spectral discretization.
By Theorem~\ref{thm:effective_kernel}, the collision invariants span the kernel of the continuous effective operator.
The spectral discretization need not annihilate these modes exactly, and zero output moments alone do not imply a zero collision response.
We therefore remove the collision-invariant component of the weighted input $r_j=g_j/\chi_j$ before each effective linear evaluation.
Using the reference equilibrium $\mathcal{N}_j=\mathcal{M}_{q,j}/(1+\tau\theta_0\mathcal{M}_{q,j})$, we compute
\begin{equation}
    \label{eq:linearized_kernel_projection}
    \left(\sum_j w_j^v\mathcal{N}_j
      \bm{\phi}_j\bm{\phi}_j^{\mathsf T}\right)\bm{b}
    =\sum_j w_j^v\mathcal{N}_j\bm{\phi}_j r_j,
    \qquad
    r_j^\perp=r_j-\bm{\phi}_j^{\mathsf T}\bm{b}.
\end{equation}
We evaluate the spectral operator $\widetilde{\mathcal{L}}_q(r^\perp)$ and apply \eqref{eq:collision_moment_correction} to its output.
The input projection ensures that the five equilibrium tangent directions are annihilated by the corrected discrete linear operator, up to roundoff.

% The corrected $\mathcal{Q}_q(f^{(k)})$ is held fixed during the inner iteration, whereas the input projection and output correction are applied at every effective linear evaluation.
% These operations are internal to the collision evaluations and impose no zero-moment constraint on $g^{(k,l)}$ itself.

For prescribed total mass $C>0$, we update the interior cells by
\begin{equation}
    \label{eq:outer_mass_normalization}
    \widetilde{f}_{i,j}
    =f_{i,j}^{(k)}-g_{i,j}^{(k)},
    \qquad
    f_{i,j}^{(k+1)}
    =\frac{C\,\widetilde{f}_{i,j}}
      {\Delta x(\Delta v)^3
       \sum_{m=1}^{N_x}\sum_n\widetilde{f}_{m,n}},
    \qquad i=1,\ldots,N_x.
\end{equation}
After normalization, we reset the wall fugacities and incoming traces by \eqref{eq:BCs} and \eqref{eq:Boundary_Cond_for_1D_QBE}.
To decide whether another outer step is needed, we evaluate the nonlinear residual at the updated distribution.
The iteration terminates when its norm does not exceed a prescribed tolerance $\varepsilon_{\mathrm{out}}>0$.

Algorithm~\ref{alg:modified_newton} summarizes the method.
All collision evaluations, including those in the residuals, use the corrections described above.
\begin{algorithm}[H]
\caption{Modified Newton method for the steady-state quantum Boltzmann equation}
\label{alg:modified_newton}
\small
\begin{algorithmic}[1]
\Require Initial interior distribution $f^{(0)}$ of total mass $C$, wall data, and tolerances $\varepsilon_{\mathrm{out}},\varepsilon_{\mathrm{abs}},\varepsilon_{\mathrm{rel}}$.
\State Apply the zero-flux construction \eqref{eq:BCs} to $f^{(0)}$ and set its incoming traces by \eqref{eq:Boundary_Cond_for_1D_QBE}.
\For{$k=0,1,\ldots$}
    \State Match $\mathcal{M}_q^{(k)}$ to the moments of $f^{(k)}$; construct and freeze $\mathcal{N}^{(k)}$, $\widetilde{\mathcal{L}}_q^{(k)}$, and $\nu_i$.
    \State Evaluate $\mathcal{Q}_q(f^{(k)})$ and correct its moments by \eqref{eq:collision_moment_correction}.
    \State Set $g^{(k,0)}\gets0$, $l\gets0$, and compute $\mathcal{E}_{k,0}$ by \eqref{eq:inner_residual}.
    \If{$\mathcal{E}_{k,0}\le\varepsilon_{\mathrm{out}}$}
        \State \Return $f^{(k)}$.
    \EndIf
    \Repeat
        \State Evaluate $\widetilde{\mathcal{L}}_q^{(k)}[g^{(k,l)}]$ with input projection \eqref{eq:linearized_kernel_projection} and output correction \eqref{eq:collision_moment_correction}.
        \State Set the incoming correction traces by \eqref{eq:Left_BC} and its right-wall analogue.
        \State Compute $g^{(k,l+1)}$ by the directional sweeps \eqref{eq:upwind}, including the local update for $v_{1,j}=0$.
        \State Update the wall fugacities from $f^{(k)}-g^{(k,l+1)}$ by \eqref{eq:BCs} and synchronize the incoming correction traces.
        \State Set $l\gets l+1$ and compute $\mathcal{E}_{k,l}$ by \eqref{eq:inner_residual}.
    \Until{$\mathcal{E}_{k,l}\le\max\{\varepsilon_{\mathrm{abs}},\varepsilon_{\mathrm{rel}}\mathcal{E}_{k,0}\}$}
    \State Set $g^{(k)}\gets g^{(k,l)}$ and form $\widetilde{f}\gets f^{(k)}-g^{(k)}$ in the interior cells.
    \State Normalize $\widetilde{f}$ to total mass $C$ by \eqref{eq:outer_mass_normalization} to obtain $f^{(k+1)}$.
    \State Apply \eqref{eq:BCs} to $f^{(k+1)}$ and reset its incoming traces by \eqref{eq:Boundary_Cond_for_1D_QBE}.
\EndFor
\end{algorithmic}
\end{algorithm}

\section{Numerical Test}\label{sec:numerical}

In this section, we first validate the effective linear operator $\Lq$ constructed in Section~\ref{sec:linear-approx}. We then compare the corresponding linear model with the full quantum Boltzmann equation and the quantum BGK model for spatially homogeneous relaxation. Finally, we apply the modified Newton solver to steady one-dimensional Couette and Fourier flows. Both Bose ($\tau=+1$) and Fermi ($\tau=-1$) gases are considered.

Unless otherwise stated, we truncate the velocity domain to $\mathcal D_L=[-L,L]^3$ and periodize the distribution function on $\mathcal D_L$. Following the de-aliasing criterion of~\cite{yin2026fast}, we choose the relative-velocity cutoff as
\begin{equation}
    R=\frac{4L}{3+\sqrt2},
    \label{eq:Rcut}
\end{equation}
so that the support of the collision integral is contained in the periodized box. We use a uniform velocity grid with $N_v$ points in each direction, an $M$-node Lebedev quadrature on $\mathbb S^2$~\citep{LebedevLaikov1999}, and $J=N_v$ radial quadrature nodes.

Throughout this section, $\mathcal Q_q$ denotes the full discrete quantum collision operator evaluated by the fast spectral method of~\cite{hu2012fast,wu2019fast} (\texttt{FullQBE}), and $\Lq$ denotes the effective linear operator. The quantum equilibrium
\begin{equation}
    \mathcal M_q(\boldsymbol v)
    =\frac{1}{\theta_0}\,
    \frac{1}{z^{-1}\exp\!\big(|\boldsymbol v-\boldsymbol u|^2/(2T)\big)-\tau}
    \label{eq:Mq}
\end{equation}
is determined by matching the moments of the input distribution.

\FloatBarrier
\subsection{Operator-Level Validation}
\label{subsec:operator-validation}
We first compare the effective operator $\widetilde{\mathcal{L}}_q$ with the exact linearization $\mathrm{D}_f\mathcal{Q}_q[\mathcal{M}_q]$ at quantum equilibrium.
We consider two values of the rescaled Planck constant: $\hbar=0.3$, for which $\theta_0=\hbar^3=0.027$ and the gas is close to the classical regime, and $\hbar=1.8$, for which $\theta_0=5.832$ and quantum effects are significant.

For each input $f_0$, set $g=f_0-\mathcal{M}_q$ and
$r=g/\bigl(\mathcal{M}_q(1+\tau\theta_0\mathcal{M}_q)\bigr)$.
Following \eqref{eq:modified_Frechet_operator}, we use the first variation
\begin{equation}
    \mathcal{L}_q(r)(\bm{v})
    =
    \mathrm{D}_f\mathcal{Q}_q[\mathcal{M}_q](g)(\bm{v}),
    \label{eq:dF}
\end{equation}
as the reference response, computed from the full quantum collision operator in the physical direction $g$.
The effective response is $\widetilde{\mathcal{L}}_q(r)=\widetilde{\mathcal{L}}_q[g]$.
We compare these responses using the discrete $L^2$ metrics
\begin{equation}
    \begin{aligned}
    \mathcal{E}_{\mathrm{abs}}
    =\big\|\mathcal{L}_q(r)-\widetilde{\mathcal{L}}_q(r)\big\|_2,\quad
    \mathcal{E}_{\mathrm{rel}}=\frac{\mathcal{E}_{\mathrm{abs}}}{\|\mathcal{L}_q(r)\|_2},\quad\cos\alpha
    =\frac{\langle \mathcal{L}_q(r),\widetilde{\mathcal{L}}_q(r)\rangle_2}{\|\mathcal{L}_q(r)\|_2\,\|\widetilde{\mathcal{L}}_q(r)\|_2}.
    \end{aligned}
    \label{eq:operator-metrics}
\end{equation}
We use $\cos\alpha$ to distinguish an overall scaling error from a change in the velocity dependence of the collision response.
For nonzero responses with $\cos\alpha>0$,
\begin{equation*}
    \min_{c>0}
    \frac{\|\mathcal{L}_q(r)-c\widetilde{\mathcal{L}}_q(r)\|_2}
         {\|\mathcal{L}_q(r)\|_2}
    =\sqrt{1-\cos^2\alpha}.
\end{equation*}
Thus the angle measures the relative error remaining after optimal positive rescaling for the chosen perturbation.
A cosine close to one suggests that rescaling may improve the approximation, but does not establish a common factor for different perturbations.
For both inputs below, we use $L=5$ and the velocity/angular resolutions $(N_v,M)=(24,26)$, $(32,38)$ and $(48,50)$.

\paragraph{Far-from-equilibrium input}
We first consider the bimodal input
\begin{equation}
    f_0(\bm{v})=\frac{\rho_0}{2(2\pi T_0)^{3/2}}
    \left(
    \exp\!\left(-\frac{|\bm{v}-\bm{u}_0|^2}{2T_0}\right)
    +
    \exp\!\left(-\frac{|\bm{v}+\bm{u}_0|^2}{2T_0}\right)
    \right),
    \label{eq:bimodal-input}
\end{equation}
with $\rho_0=1$, $T_0=3/8$ and $\bm{u}_0=(1,\,0.5,\,0)$, a standard far-from-equilibrium configuration in the quantum Boltzmann literature~\citep{hu2012fast}.
Both responses are evaluated at the equilibrium $\mathcal{M}_q$ with the same moments as $f_0$.
Table~\ref{tab:far-equilibrium-operator-comparison} reports the results.

\begin{table}[pos=!htbp]
\centering
\caption{Comparison of the exact and effective responses $\mathcal{L}_q(r)$ and $\widetilde{\mathcal{L}}_q(r)$ for the bimodal input \eqref{eq:bimodal-input}.}
\label{tab:far-equilibrium-operator-comparison}
\footnotesize
\begin{tabular}{@{}ccccccc@{}}
\toprule
$(N_v,M)$ & Parameters & $\|\mathcal{L}_q(r)\|_2$ & $\|\widetilde{\mathcal{L}}_q(r)\|_2$ & $\mathcal{E}_{\mathrm{abs}}$ & $\mathcal{E}_{\mathrm{rel}}$ & $\cos\alpha$ \\
\midrule
$(24,26)$
 & $\hbar=0.30,\ \tau=+1$ & $6.81\times10^{-1}$ & $6.76\times10^{-1}$ & $3.09\times10^{-2}$ & $4.54\%$ & $0.998986$ \\
 & $\hbar=0.30,\ \tau=-1$ & $6.82\times10^{-1}$ & $6.77\times10^{-1}$ & $3.11\times10^{-2}$ & $4.56\%$ & $0.998975$ \\
 & $\hbar=1.80,\ \tau=+1$ & $6.70\times10^{-1}$ & $6.32\times10^{-1}$ & $7.74\times10^{-2}$ & $11.54\%$ & $0.994710$ \\
 & $\hbar=1.80,\ \tau=-1$ & $1.04\times10^{0}$  & $1.11\times10^{0}$  & $1.63\times10^{-1}$ & $15.71\%$ & $0.990259$ \\
\midrule
$(32,38)$
 & $\hbar=0.30,\ \tau=+1$ & $6.76\times10^{-1}$ & $6.76\times10^{-1}$ & $5.81\times10^{-3}$ & $0.86\%$ & $0.999963$ \\
 & $\hbar=0.30,\ \tau=-1$ & $6.76\times10^{-1}$ & $6.77\times10^{-1}$ & $5.88\times10^{-3}$ & $0.87\%$ & $0.999963$ \\
 & $\hbar=1.80,\ \tau=+1$ & $6.67\times10^{-1}$ & $6.32\times10^{-1}$ & $7.42\times10^{-2}$ & $11.12\%$ & $0.994992$ \\
 & $\hbar=1.80,\ \tau=-1$ & $1.02\times10^{0}$  & $1.11\times10^{0}$  & $1.28\times10^{-1}$ & $12.51\%$ & $0.996023$ \\
\midrule
$(48,50)$
 & $\hbar=0.30,\ \tau=+1$ & $6.77\times10^{-1}$ & $6.76\times10^{-1}$ & $2.37\times10^{-3}$ & $0.35\%$ & $0.999994$ \\
 & $\hbar=0.30,\ \tau=-1$ & $6.77\times10^{-1}$ & $6.77\times10^{-1}$ & $2.37\times10^{-3}$ & $0.35\%$ & $0.999994$ \\
 & $\hbar=1.80,\ \tau=+1$ & $6.68\times10^{-1}$ & $6.32\times10^{-1}$ & $7.41\times10^{-2}$ & $11.10\%$ & $0.995045$ \\
 & $\hbar=1.80,\ \tau=-1$ & $1.03\times10^{0}$  & $1.11\times10^{0}$  & $1.17\times10^{-1}$ & $11.35\%$ & $0.996462$ \\
\bottomrule
\end{tabular}
\end{table}
\FloatBarrier

\paragraph{Near-equilibrium input}
We next consider a small perturbation of equilibrium, a regime relevant to the linear solves in the modified Newton method.
For each $(\hbar,\tau)$, we take $\mathcal{M}_q$ to be the equilibrium of the bimodal input \eqref{eq:bimodal-input}, with $\rho=1$, $\bm{u}=\bm{0}$ and $\theta=19/24$.
The parameters $(z,T)$ are determined by \eqref{eq:fugality_and_temperature}.
We set
\begin{equation}
    f_0(\bm{v})=\mathcal{M}_q(\bm{v})\,
    \big(1+\varepsilon\,\varphi(\bm{v})\big),\qquad
    \varphi(\bm{v})=\frac{(v_1-u_1)^2-(v_2-u_2)^2}{T+|\bm{v}-\bm{u}|^2},
    \qquad \varepsilon=0.05.
    \label{eq:near-equilibrium-input}
\end{equation}
This perturbation leaves the mass, momentum and energy unchanged, so the moment-matched equilibrium remains $\mathcal{M}_q$.
Table~\ref{tab:near-equilibrium-operator-comparison} reports the results.

\begin{table}[pos=!htbp]
\centering
\caption{Comparison of the exact and effective responses for the near-equilibrium input \eqref{eq:near-equilibrium-input}, with $\varepsilon=0.05$.}
\label{tab:near-equilibrium-operator-comparison}
\footnotesize
\begin{tabular}{@{}ccccccc@{}}
\toprule
$(N_v,M)$ & Parameters & $\|\mathcal{L}_q(r)\|_2$ & $\|\widetilde{\mathcal{L}}_q(r)\|_2$ & $\mathcal{E}_{\mathrm{abs}}$ & $\mathcal{E}_{\mathrm{rel}}$ & $\cos\alpha$ \\
\midrule
$(24,26)$
 & $\hbar=0.30,\ \tau=+1$ & $1.08\times10^{-2}$ & $1.08\times10^{-2}$ & $2.41\times10^{-4}$ & $2.22\%$ & $0.999753$ \\
 & $\hbar=0.30,\ \tau=-1$ & $1.08\times10^{-2}$ & $1.08\times10^{-2}$ & $2.41\times10^{-4}$ & $2.22\%$ & $0.999753$ \\
 & $\hbar=1.80,\ \tau=+1$ & $1.16\times10^{-2}$ & $1.06\times10^{-2}$ & $1.47\times10^{-3}$ & $12.63\%$ & $0.995753$ \\
 & $\hbar=1.80,\ \tau=-1$ & $9.99\times10^{-3}$ & $1.13\times10^{-2}$ & $1.61\times10^{-3}$ & $16.15\%$ & $0.995686$ \\
\midrule
$(32,38)$
 & $\hbar=0.30,\ \tau=+1$ & $1.08\times10^{-2}$ & $1.08\times10^{-2}$ & $5.73\times10^{-5}$ & $0.53\%$ & $0.999986$ \\
 & $\hbar=0.30,\ \tau=-1$ & $1.08\times10^{-2}$ & $1.08\times10^{-2}$ & $5.67\times10^{-5}$ & $0.52\%$ & $0.999986$ \\
 & $\hbar=1.80,\ \tau=+1$ & $1.16\times10^{-2}$ & $1.06\times10^{-2}$ & $1.44\times10^{-3}$ & $12.40\%$ & $0.996001$ \\
 & $\hbar=1.80,\ \tau=-1$ & $9.98\times10^{-3}$ & $1.13\times10^{-2}$ & $1.59\times10^{-3}$ & $15.93\%$ & $0.996058$ \\
\midrule
$(48,50)$
 & $\hbar=0.30,\ \tau=+1$ & $1.08\times10^{-2}$ & $1.08\times10^{-2}$ & $2.12\times10^{-5}$ & $0.20\%$ & $0.999998$ \\
 & $\hbar=0.30,\ \tau=-1$ & $1.08\times10^{-2}$ & $1.08\times10^{-2}$ & $1.94\times10^{-5}$ & $0.18\%$ & $0.999998$ \\
 & $\hbar=1.80,\ \tau=+1$ & $1.16\times10^{-2}$ & $1.06\times10^{-2}$ & $1.44\times10^{-3}$ & $12.39\%$ & $0.996010$ \\
 & $\hbar=1.80,\ \tau=-1$ & $9.98\times10^{-3}$ & $1.13\times10^{-2}$ & $1.59\times10^{-3}$ & $15.92\%$ & $0.996075$ \\
\bottomrule
\end{tabular}
\end{table}
\FloatBarrier

Tables~\ref{tab:far-equilibrium-operator-comparison} and~\ref{tab:near-equilibrium-operator-comparison} show two distinct error regimes.
For $\hbar=0.3$, the relative error decreases under refinement, and the Bose and Fermi results are nearly identical.
This behavior is consistent with the classical limit described in Section~\ref{sec:linear-approx} and the classical linearized operator of~\cite{yin2026fast}.
For $\hbar=1.8$, refinement produces much smaller changes, and the errors remain appreciable on the finest grid.
This suggests that the radial approximation is the main source of the remaining discrepancy.
The approximation preserves the conservation and linearized entropy structure established in Section~\ref{sec:linear-approx}, but does not reproduce the exact collision response.
The direction cosine exceeds $0.99$ in every case, indicating strong directional agreement despite the differences in response magnitude.

In the modified Newton method of Section~\ref{sec:newton}, the nonlinear residual is evaluated with the full collision operator, while $\widetilde{\mathcal{L}}_q$ is used only in the correction equation.
The radial approximation therefore affects convergence but does not alter the steady-state equation satisfied when the full nonlinear residual vanishes.

\paragraph{Efficiency}
We next compare the cost of evaluating the full collision operator \texttt{FullQBE} and the effective operator $\widetilde{\mathcal{L}}_q$.
Table~\ref{tab:full-linearized-qbe-time-cost} reports the wall-clock time per evaluation, with $J=N_v$ for both methods.
Figure~\ref{fig:linearized-qbe-scaling} shows the measured cost of $\widetilde{\mathcal{L}}_q$ for $N_v$ from $24$ to $128$, together with the reference curve $CN_v^4\log N_v$.

\begin{table}[pos=!htbp]
\centering
\caption{Wall-clock time per collision evaluation. Both methods use $J=N_v$; \texttt{FullQBE} uses $M\approx N_v$. For $N_v=32,48,64$, the angular resolutions are $M=38,50,74$, respectively, as in Tables~\ref{tab:far-equilibrium-operator-comparison} and~\ref{tab:near-equilibrium-operator-comparison}. Timings for the effective operator are averaged over $10$ calls at each resolution; \texttt{FullQBE} uses $5$ calls at $N_v=32$ and one call at $N_v=48,64$. For each method, $t_{N_v}$ denotes the mean wall time, and the ratios $t_{48}/t_{32}$ and $t_{64}/t_{48}$ measure its growth under refinement.}
\label{tab:full-linearized-qbe-time-cost}
\footnotesize
\begin{tabular}{@{}cccccc@{}}
\toprule
Method & $N_v=32$ (s) & ratio & $N_v=48$ (s) & ratio & $N_v=64$ (s) \\
\midrule
\texttt{FullQBE} total time    & $48.83$ & $6.02$ & $294.14$ & $4.92$ & $1445.95$ \\
$\widetilde{\mathcal{L}}_q$ time   & $0.29$  & $3.00$ & $0.87$   & $2.61$ & $2.27$ \\
Speedup               & $169.24$ &       & $338.09$ &        & $638.37$ \\
\bottomrule
\end{tabular}
\end{table}
\FloatBarrier

Over the tested range, the measured growth is consistent with the $O(N_v^4\log N_v)$ complexity derived in Section~\ref{sec:linear-approx}.
The $O(MN_v^4\log N_v)$ cost of \texttt{FullQBE}~\citep{hu2012fast} includes an additional angular quadrature factor.
With $M\approx N_v$, this factor increases the cost difference under refinement.
The measured speedup rises from about $169$ at $N_v=32$ to about $638$ at $N_v=64$, giving a reduction of two to three orders of magnitude in the cost per evaluation.

\begin{figure}[pos=!htbp]
    \centering
    % Relabel the raster figure's wall-time symbol in TeX; preserve all plotted data.
    \setlength{\unitlength}{0.0062\linewidth}
    \begin{picture}(100,70.67)
        \put(0,0){\includegraphics[width=100\unitlength]{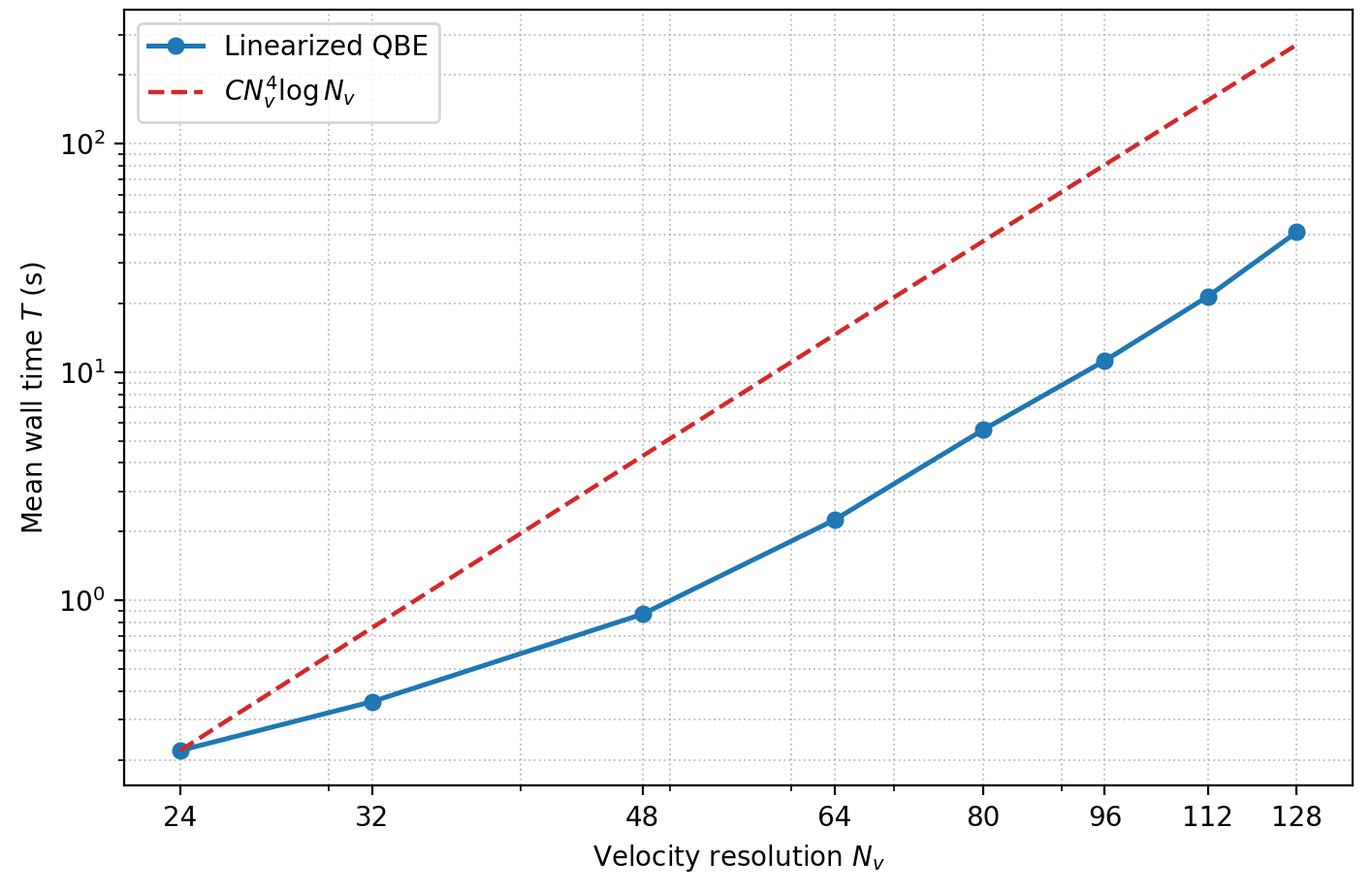}}
        \put(0,21.2){\color{white}\rule{4\unitlength}{28.4\unitlength}}
        \put(2,35.3){\makebox(0,0){\rotatebox{90}{\footnotesize Mean wall time $t$ (s)}}}
    \end{picture}
    \caption{Wall-clock time per evaluation of $\widetilde{\mathcal{L}}_q$ for $J=N_v$, compared with the reference curve $CN_v^4\log N_v$.}
    \label{fig:linearized-qbe-scaling}
\end{figure}
\FloatBarrier

\subsection{Spatially Homogeneous QBE}
\label{sec:spatial_homo}
We next compare the effective linear model with the full quantum Boltzmann equation and the quantum BGK model for homogeneous relaxation.
We solve
\begin{equation}
    \partial_t f(t,\bm v)=\mathcal C(f)(t,\bm v),
    \qquad \bm v\in\mathbb R^3,
    \label{eq:homogeneous-problem}
\end{equation}
with
\begin{equation*}
    \mathcal C(f)=
    \begin{cases}
        \mathcal Q_q(f), & \text{full quantum Boltzmann model},\\[1mm]
        \Lq[f-\mathcal M_q], & \text{effective linear model},\\[1mm]
        \nu(\mathcal M_q-f), & \text{quantum BGK model}.
    \end{cases}
\end{equation*}
The square-bracket notation for $\Lq$ is defined in \eqref{eq:modified_Frechet_operator}.
For the BGK model~\citep{filbet2010numerical}, we fix $\nu=\max\{10^{-3},\rho_{\mathrm d,0}\}$, where
$\rho_{\mathrm d,0}=(\Delta v)^3\sum_j f_0(\bm v_j)$ is the discrete initial density and $\Delta v=2L/N_v$.

Both initial data below satisfy
\begin{equation}
    \int_{\mathbb R^3}f_0(\bm v)\,\dif\bm v=1,\qquad
    \int_{\mathbb R^3}\bm v f_0(\bm v)\,\dif\bm v=\bm 0,\qquad
    \int_{\mathbb R^3}|\bm v|^2f_0(\bm v)\,\dif\bm v=3.
    \label{eq:homogeneous-moments}
\end{equation}
For each statistic, the continuous models therefore share the same conserved moments and equilibrium \eqref{eq:Mq}.
In the computations, equilibrium parameters are recovered from discrete moments.
The equilibrium in the linear model is fixed by the initial data; in the BGK model, it is recomputed and corrected to match the moments at each Runge--Kutta stage.

We use $R=6$, $L=(3+\sqrt2)R/4\approx6.6213$, $(N_v,M)=(32,38)$ and $\hbar=1.8$ ($\theta_0=5.832$).
Time integration uses the classical fourth-order Runge--Kutta scheme with $\Delta t=0.1$ on $0\le t\le6$.
For each model and statistic, we measure the relative error against the full QBE solution on the same grid:
\begin{equation}
    \mathcal E_{\mathrm{rel}}(t)
    =
    \frac{\|f_{\mathrm{model}}(t)-f_{\mathrm{QBE}}(t)\|_2}
         {\|f_{\mathrm{QBE}}(t)\|_2},
    \label{eq:homogeneous-error}
\end{equation}
where $\|\cdot\|_2$ denotes the discrete $L^2$ norm.

\paragraph{Test 1: a smooth four-peak distribution}
The first initial datum is
\begin{equation}
    f_0^{(1)}(\bm v)
    =
    \frac1{10}\left(\frac3{2\pi}\right)^{3/2}
    \sum_{s=\pm1}
    \left[
        4\exp\!\left(-\frac32|\bm v-s\sqrt2\bm e_1|^2\right)
        +\exp\!\left(-\frac32|\bm v-s\sqrt2\bm e_2|^2\right)
    \right],
    \label{eq:homogeneous-initial-1}
\end{equation}
where $\bm e_1,\bm e_2$ are the first two Cartesian unit vectors.
The higher peaks along $v_1$ give a strongly anisotropic pressure tensor.
We use the hard-sphere kernel
\begin{equation}
    B(|\bm v-\bm v_*|)=\frac1{4\pi}|\bm v-\bm v_*|.
    \label{eq:hard-sphere-kernel}
\end{equation}

Figure~\ref{fig:case1-errors} shows the corresponding relative errors.
For the Fermi and Bose gases, respectively, the BGK errors peak at approximately $0.159$ and $0.224$ near $t=0.8$.
The linear-model peaks are approximately $0.044$ at $t=0.6$ and $0.051$ at $t=0.5$, reducing the peak errors by factors of about $3.6$ and $4.4$.
The linear model remains closer to the full QBE solution throughout the plotted evolution.

\begin{figure}[pos=!htbp]
    \centering
    \includegraphics[width=0.65\linewidth]{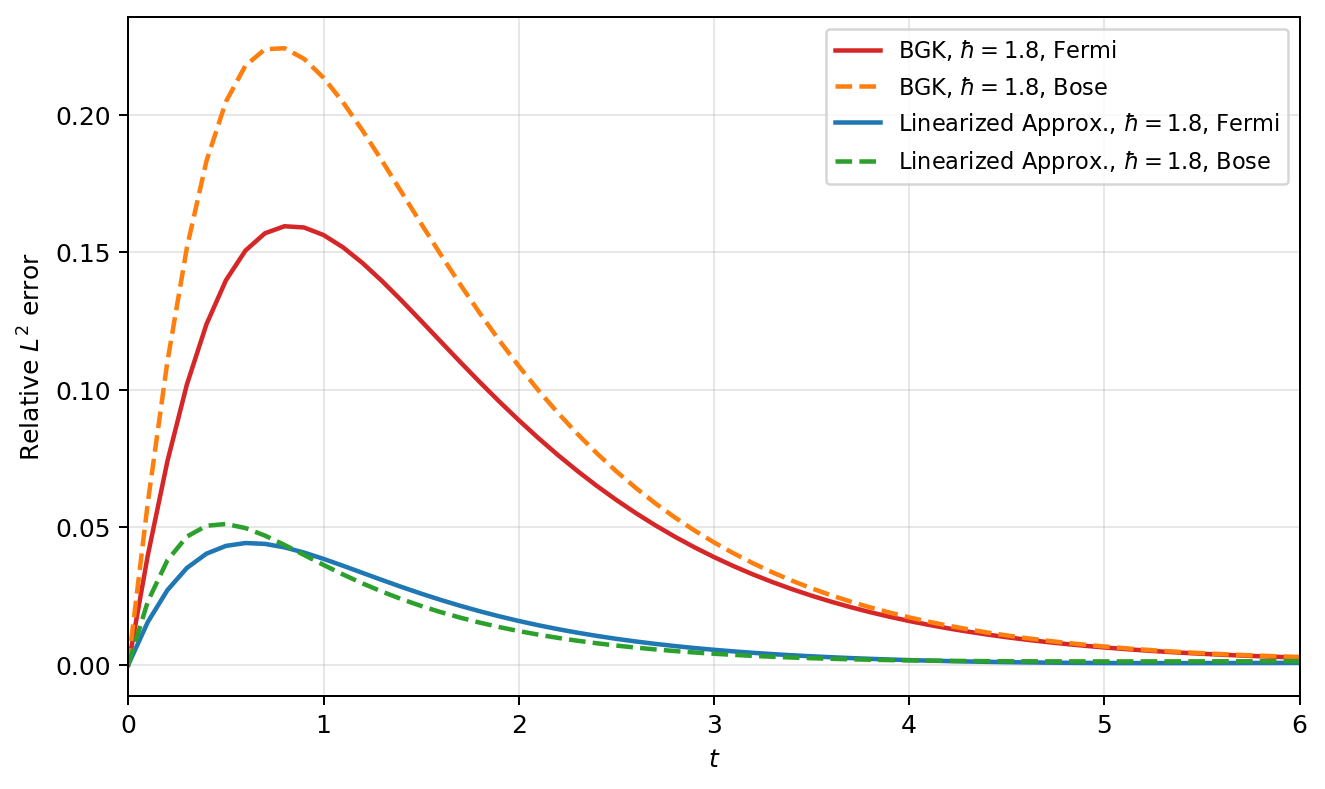}
    \caption{Test 1: relative $L^2$ errors \eqref{eq:homogeneous-error} of the effective linear and quantum BGK models against the full QBE solution, for both statistics at $(N_v,M)=(32,38)$. The legend uses $\hbar=1.8$.}
    \label{fig:case1-errors}
\end{figure}
\FloatBarrier

Figures~\ref{fig:case1-distributions-bose} and~\ref{fig:case1-distributions-fermi} compare the slices $f(t,v_1,v_2,0)$ at $t=1,2,4$ for the Bose and Fermi gases, respectively.
For both statistics, the effective linear model follows the merging of the two peaks along $v_1$ and the subsequent approach to an isotropic profile.
The BGK solution retains more pronounced peaks at $t=1$ and stronger anisotropy at $t=2$.

\begin{figure}[pos=!htbp]
    \centering
    \subfloat[Bose gas ($\tau=+1$)]{%
        \includegraphics[width=0.73\linewidth]{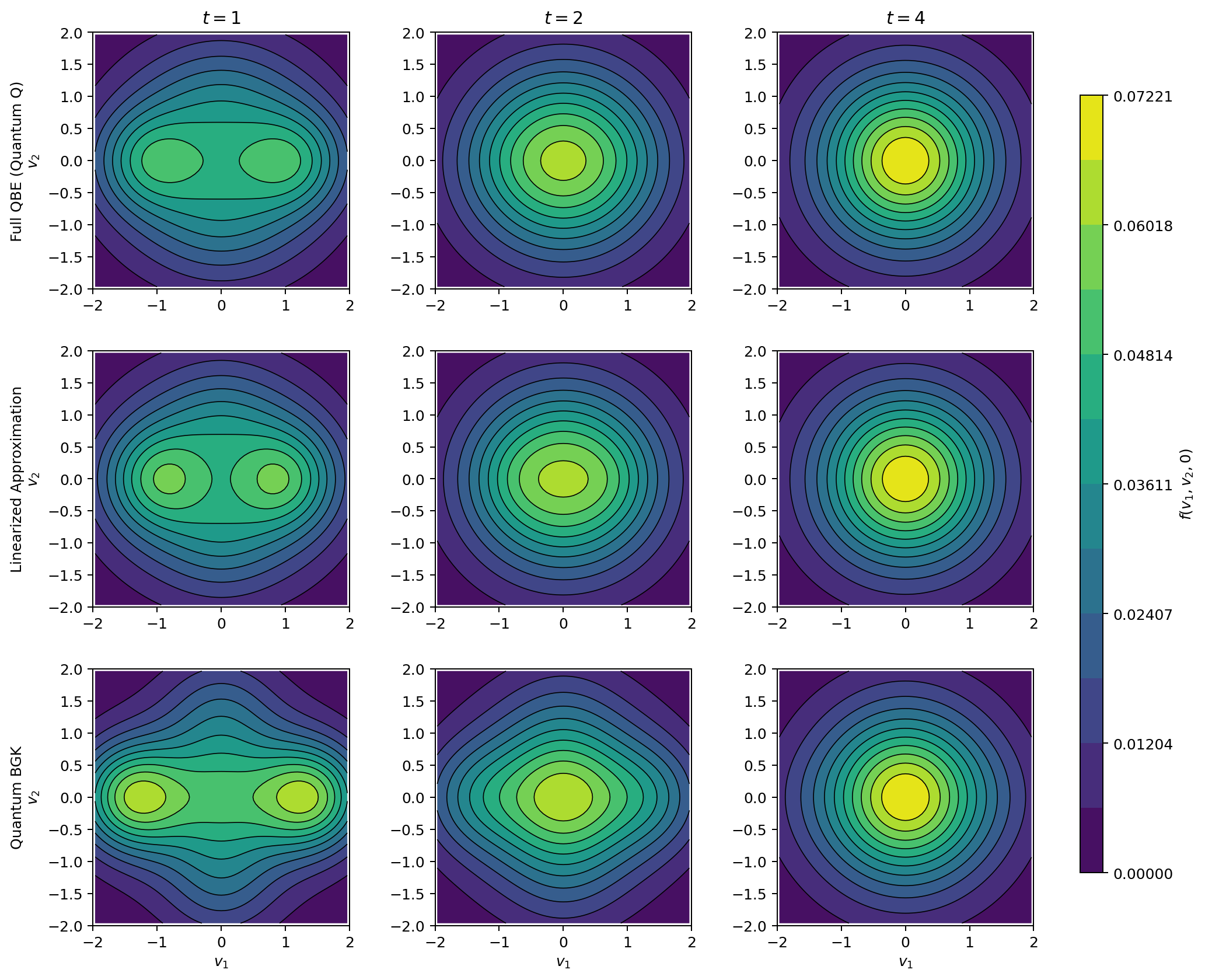}%
        \label{fig:case1-distributions-bose}}
    \par\vspace{3mm}
    \subfloat[Fermi gas ($\tau=-1$)]{%
        \includegraphics[width=0.73\linewidth]{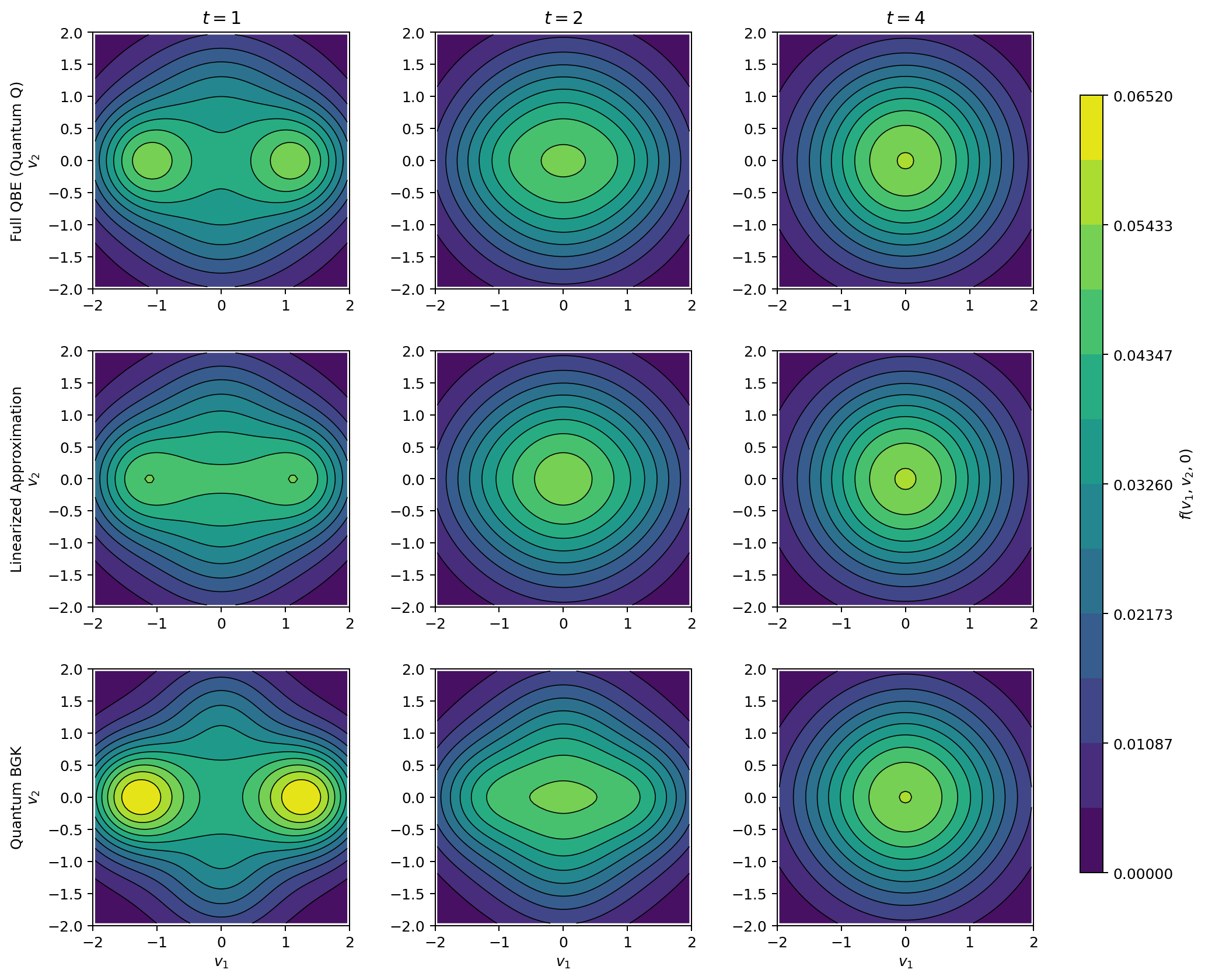}%
        \label{fig:case1-distributions-fermi}}
    \caption{Test 1: slices $f(t,v_1,v_2,0)$ for (a) the Bose gas and (b) the Fermi gas, with $\hbar=1.8$ and $(N_v,M)=(32,38)$. Columns correspond to $t=1,2,4$; rows show the full QBE, effective linear and quantum BGK solutions. Each subfigure has its own colour scale.}
    \label{fig:case1-distributions}
\end{figure}
\FloatBarrier

\paragraph{Test 2: a discontinuous two-temperature distribution}
Following the discontinuous example of~\cite{yin2026fast}, we take
\begin{equation*}
    f_0^{(2)}(\bm v)
    =
    \begin{cases}
    \displaystyle
    \frac65\left(\frac3{4\pi}\right)^{3/2}
    \exp\!\left(-\frac{3|\bm v|^2}{4}\right),
    & v_1>0,
    \\[3mm]
    \displaystyle
    \frac45\left(\frac1{3\pi}\right)^{3/2}
    \exp\!\left(-\frac{|\bm v|^2}{3}\right),
    & v_1<0,
    \end{cases}
\end{equation*}
with the arithmetic mean of the one-sided traces at $v_1=0$.
The variable hard-sphere kernel is
\begin{equation}
    B(|\bm v-\bm v_*|)=\frac1{4\pi}|\bm v-\bm v_*|^{0.56},
    \label{eq:vhs-kernel}
\end{equation}
with the exponent commonly used for argon.

The peak errors in Figure~\ref{fig:case2-errors} are smaller than in Test~1 for both models.
For the Fermi and Bose gases, respectively, the BGK errors peak at approximately $0.068$ and $0.058$, compared with $0.029$ and $0.016$ for the effective linear model.
The peak errors are thus reduced by factors of about $2.3$ and $3.6$.

\begin{figure}[pos=!htbp]
    \centering
    \includegraphics[width=0.82\linewidth]{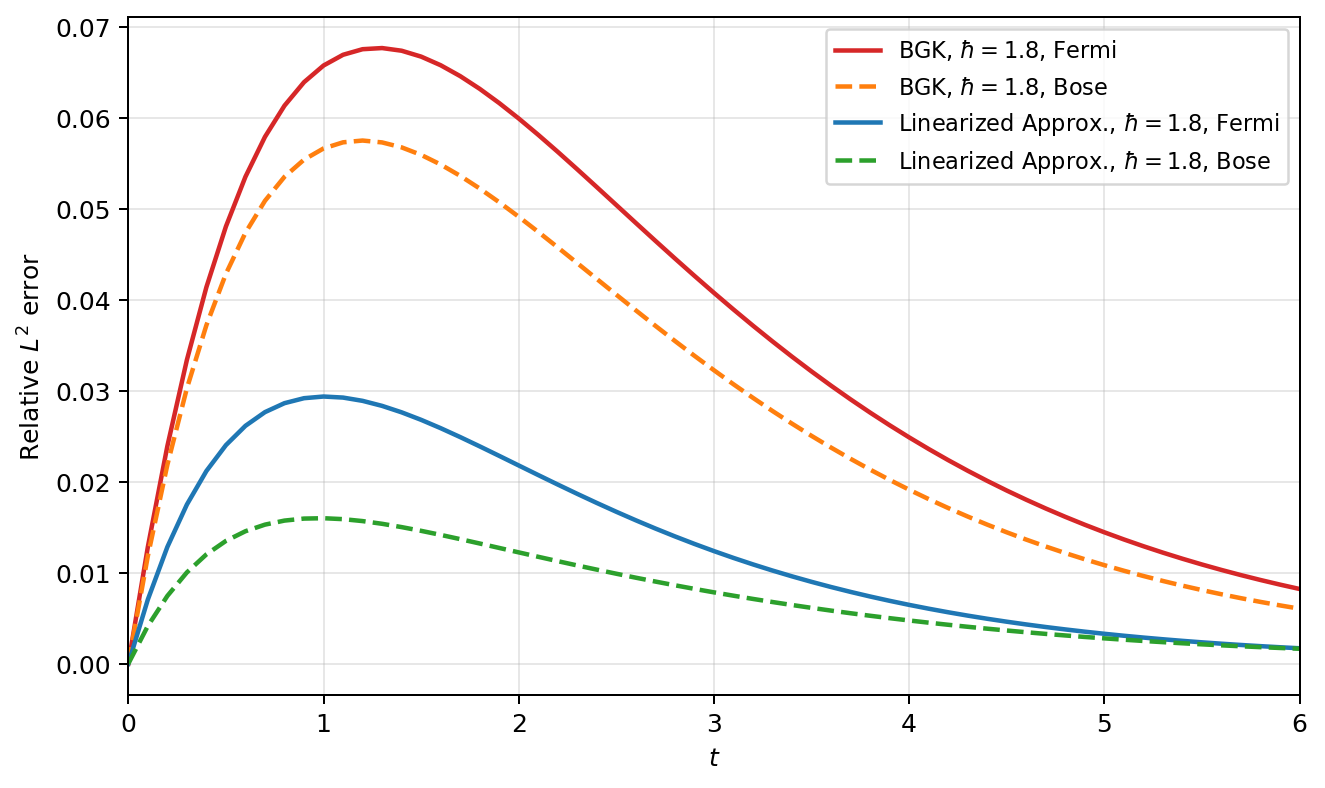}
    \caption{Test 2: relative $L^2$ errors \eqref{eq:homogeneous-error} of the effective linear and quantum BGK models against the full QBE solution, for both statistics at $\hbar=1.8$ and $(N_v,M)=(32,38)$.}
    \label{fig:case2-errors}
\end{figure}
\FloatBarrier

Figures~\ref{fig:case2-distributions-bose} and~\ref{fig:case2-distributions-fermi} show the asymmetric profiles at $t=1$ and their subsequent approach to isotropy for the Bose and Fermi gases, respectively.
For both statistics, the effective linear solution closely follows the full QBE solution, whereas the BGK solution retains more of the initial asymmetry.
The Bose profiles have a sharper central peak than the Fermi profiles.

\begin{figure}[pos=!htbp]
    \centering
    \subfloat[Bose gas ($\tau=+1$)]{%
        \includegraphics[width=0.73\linewidth]{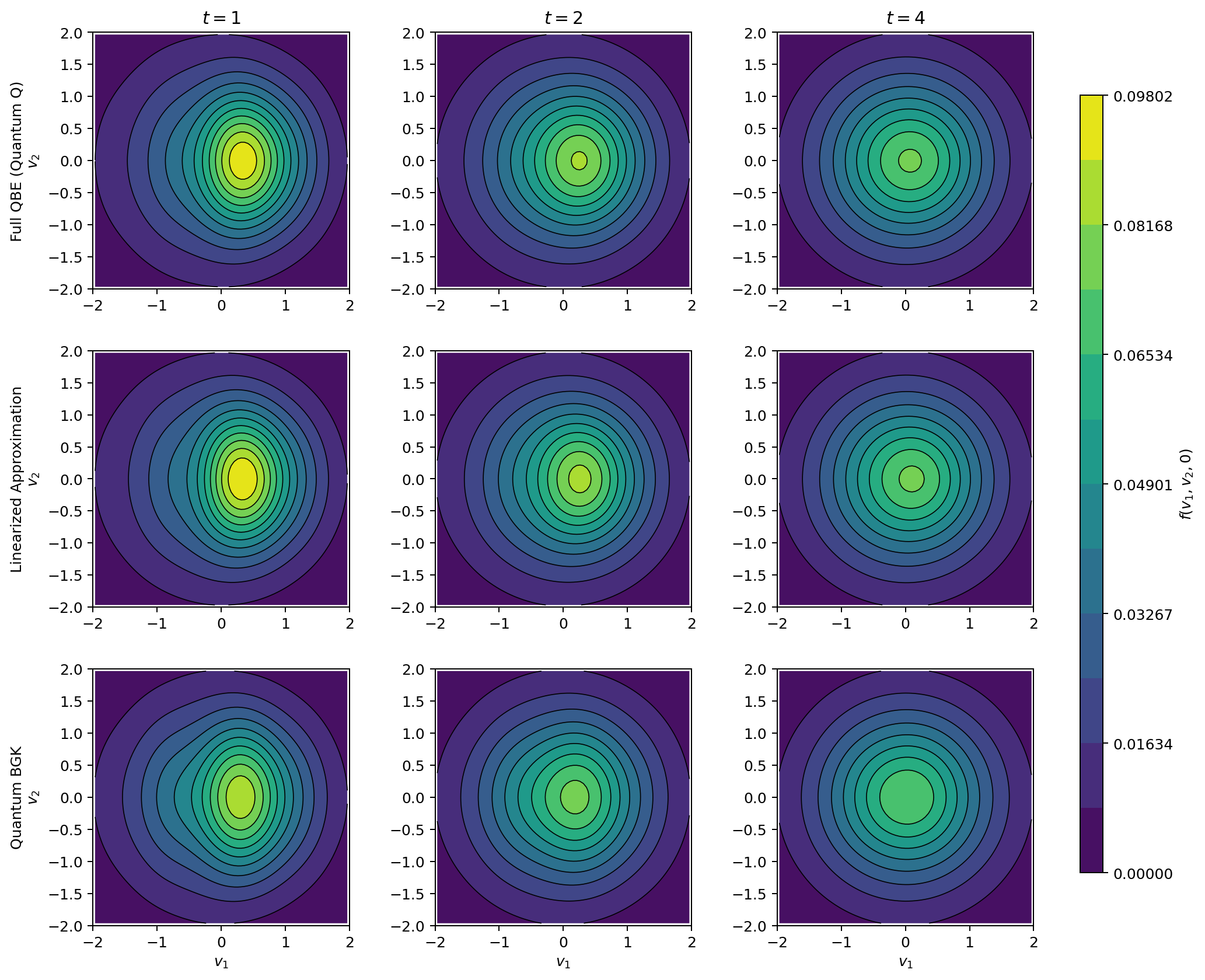}%
        \label{fig:case2-distributions-bose}}
    \par\vspace{3mm}
    \subfloat[Fermi gas ($\tau=-1$)]{%
        \includegraphics[width=0.73\linewidth]{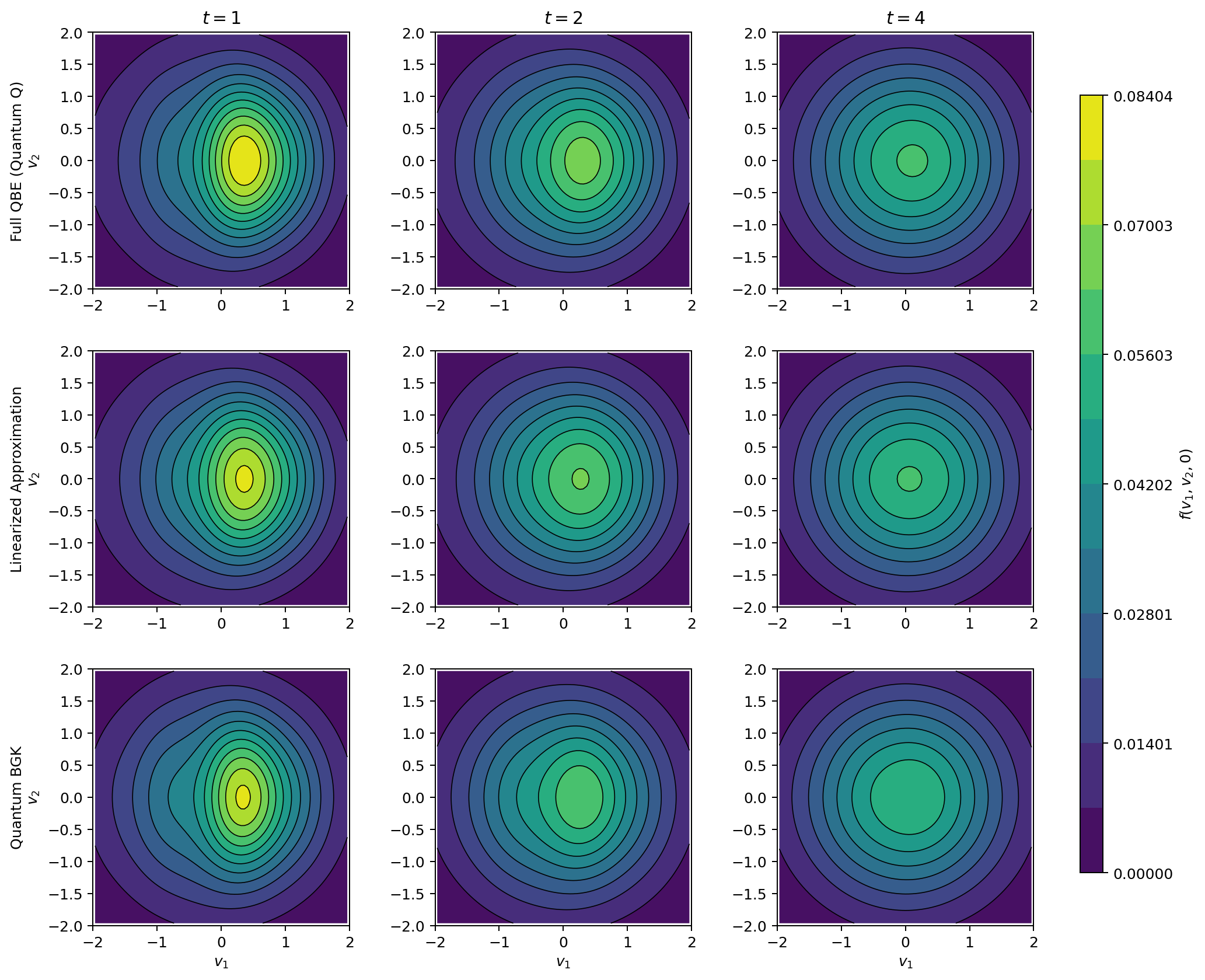}%
        \label{fig:case2-distributions-fermi}}
    \caption{Test 2: slices $f(t,v_1,v_2,0)$ for (a) the Bose gas and (b) the Fermi gas, with $\hbar=1.8$ and $(N_v,M)=(32,38)$. Columns and rows follow Figure~\ref{fig:case1-distributions}. Each subfigure has its own colour scale.}
    \label{fig:case2-distributions}
\end{figure}
\FloatBarrier

\subsection{Couette Flow}
\label{subsec:couette}
We solve the steady problem \eqref{eq:1D_Steady_QBE} on $(x_L,x_R)=(-\tfrac12,\tfrac12)$ for a gas between two plates moving in opposite tangential directions, with wall data
\begin{equation}
    \bm u_L=(0,-0.5,0),\qquad
    \bm u_R=(0,0.5,0),\qquad
    \theta_L=\theta_R=1.
    \label{eq:couette-wall-parameters}
\end{equation}
Here $\theta_L,\theta_R$ are the kinetic wall temperatures defined in \eqref{eq:momentums}.
Wall fugacities, thermodynamic temperatures and incoming traces are determined by \eqref{eq:fugality_and_temperature}, \eqref{eq:BCs} and \eqref{eq:Boundary_Cond_for_1D_QBE}, as in Algorithm~\ref{alg:modified_newton}.

For prescribed total mass $C=1$, the initial state at each interior cell centre $x_i$ is
\begin{equation}
    \label{eq:IV}
    f_{i,j}^{(0)}
    =\frac{1}{\theta_0}
     \frac{1}{(z^{(0)})^{-1}
     \exp\!\left(\frac{|\bm v_j-\bm u^{(0)}(x_i)|^2}{2T^{(0)}}\right)-\tau},\quad
    \bm u^{(0)}(x_i)
    =\frac{x_R-x_i}{x_R-x_L}\bm u_L
     +\frac{x_i-x_L}{x_R-x_L}\bm u_R.
\end{equation}
The constants $z^{(0)},T^{(0)}$ follow from \eqref{eq:fugality_and_temperature} with $\rho^{(0)}=C/(x_R-x_L)=1$ and $\theta^{(0)}=\theta_L=\theta_R=1$:
\begin{equation}
    \label{eq:initial_param}
    \frac{\bigl[\tau\operatorname{Li}_{\frac32}(\tau z^{(0)})\bigr]^{\frac52}}
         {\bigl[\tau\operatorname{Li}_{\frac52}(\tau z^{(0)})\bigr]^{\frac32}}
    =\frac{\theta_0}{(2\pi)^{\frac32}},
    \qquad
    T^{(0)}
    =\frac{\operatorname{Li}_{\frac32}(\tau z^{(0)})}
          {\operatorname{Li}_{\frac52}(\tau z^{(0)})}.
\end{equation}

We use the Maxwell-molecule kernel $B(|\bm v-\bm v_*|)=1/(4\pi)$, cutoff $R=8$ with $L$ given by \eqref{eq:Rcut}, and resolution $(N_x,N_v,M)=(50,48,50)$.
Algorithm~\ref{alg:modified_newton} uses the residual norm \eqref{eq:inner_residual} and stopping tolerances
\begin{equation}
    \label{eq:error_settings}
    \varepsilon_{\mathrm{out}}=10^{-5},\qquad
    \varepsilon_{\mathrm{rel}}=10^{-3},\qquad
    \varepsilon_{\mathrm{abs}}=10^{-7}.
\end{equation}
The runs cover $\Kn\in\{0.1,1,10\}$, $\tau=\pm1$ and $\hbar\in\{1.2,1.8\}$.

\paragraph{Steady profiles}
Figure~\ref{fig:couette-profiles} shows the tangential velocity $u_2$, kinetic temperature $\theta$, and shear stress $P_{12}$, the $(1,2)$ entry of $\mathbb P$ in \eqref{eq:momentums}.
The black dashed curves are references obtained by solving the classical Boltzmann equation with the same prescribed wall velocities and kinetic temperatures.
The velocity is antisymmetric about the channel centre.
As $\Kn$ decreases, its magnitude increases towards the wall speed, and the velocity slip decreases.
Relative to the classical reference, the Bose profiles have larger velocity magnitudes and the Fermi profiles smaller ones; the separation increases with $\hbar$ and is less pronounced at $\Kn=0.1$.

The temperature is symmetric and reaches its maximum at the channel centre, with $\theta>1$ throughout the gas, consistent with viscous heating.
At $\Kn=10$ the profiles are nearly flat.
Decreasing $\Kn$ lowers the temperature excess above the walls but increases its spatial variation.
For $\Kn=10$ and $1$, the Bose temperature lies above the classical reference and the Fermi temperature below it; this ordering reverses at $\Kn=0.1$.
Thus the effect of quantum statistics on the temperature profile depends on the rarefaction level.

The shear stress is negative throughout the channel, consistent with momentum transport opposing the imposed shear.
Its magnitude decreases as $\Kn$ decreases and is smaller for Bose gases and larger for Fermi gases than for the classical reference, with a stronger separation at $\hbar=1.8$.
The profiles are nearly uniform for $\Kn=10$ and $1$, in agreement with $\partial_{x_1}P_{12}=0$ from the steady tangential-momentum balance \eqref{eq:Quantum_Euler} with $u_1=0$.
The visible spatial variation at $\Kn=0.1$ is therefore a numerical defect in this cell-centred moment.

\begin{figure}[pos=!htbp]
    \centering
    \begin{minipage}{0.98\linewidth}
        \small (a) Tangential velocity\par
        \includegraphics[width=\linewidth]{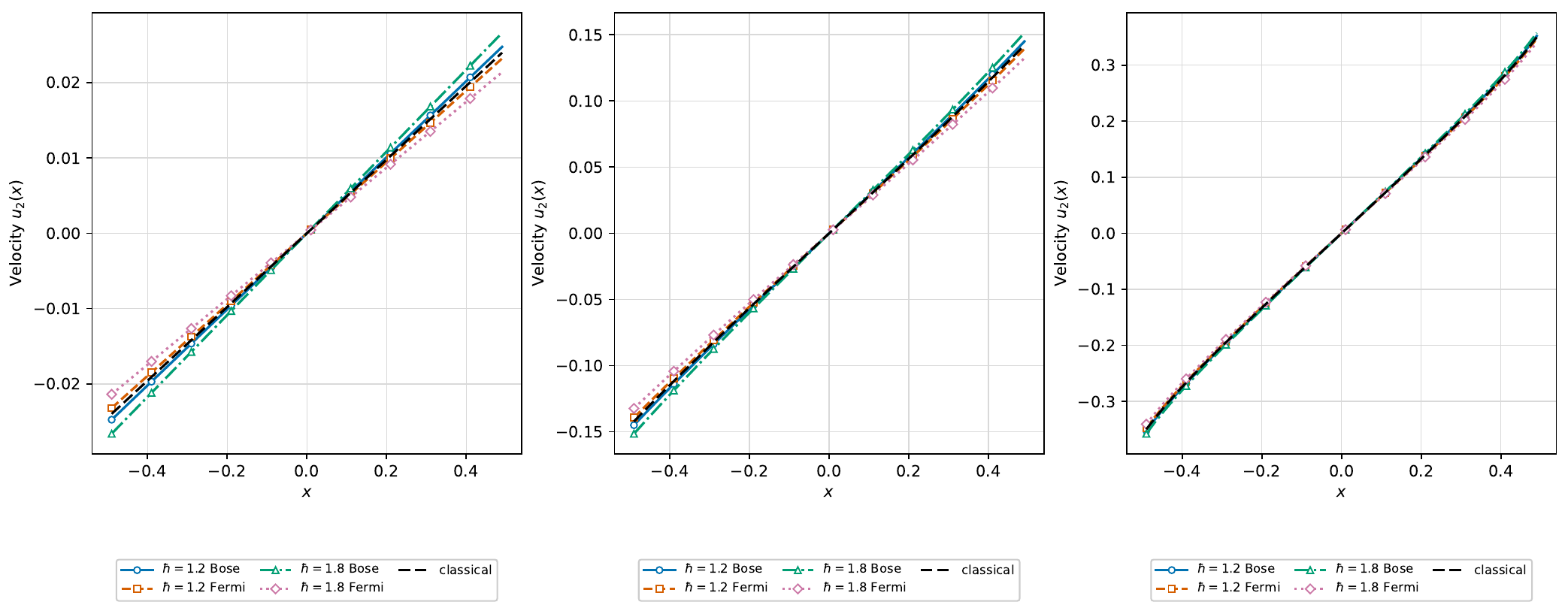}\par\vspace{2mm}
        \small (b) Kinetic temperature\par
        \includegraphics[width=\linewidth]{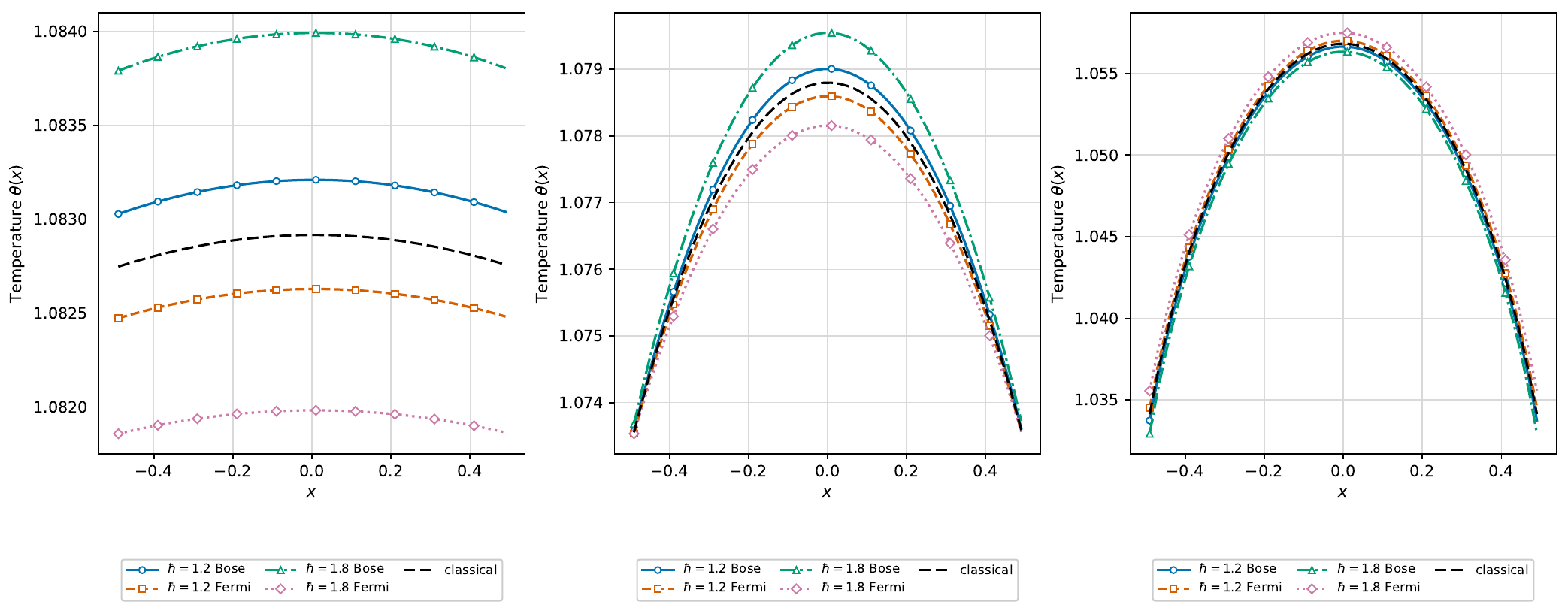}\par\vspace{2mm}
        \small (c) Shear stress\par
        \includegraphics[width=\linewidth]{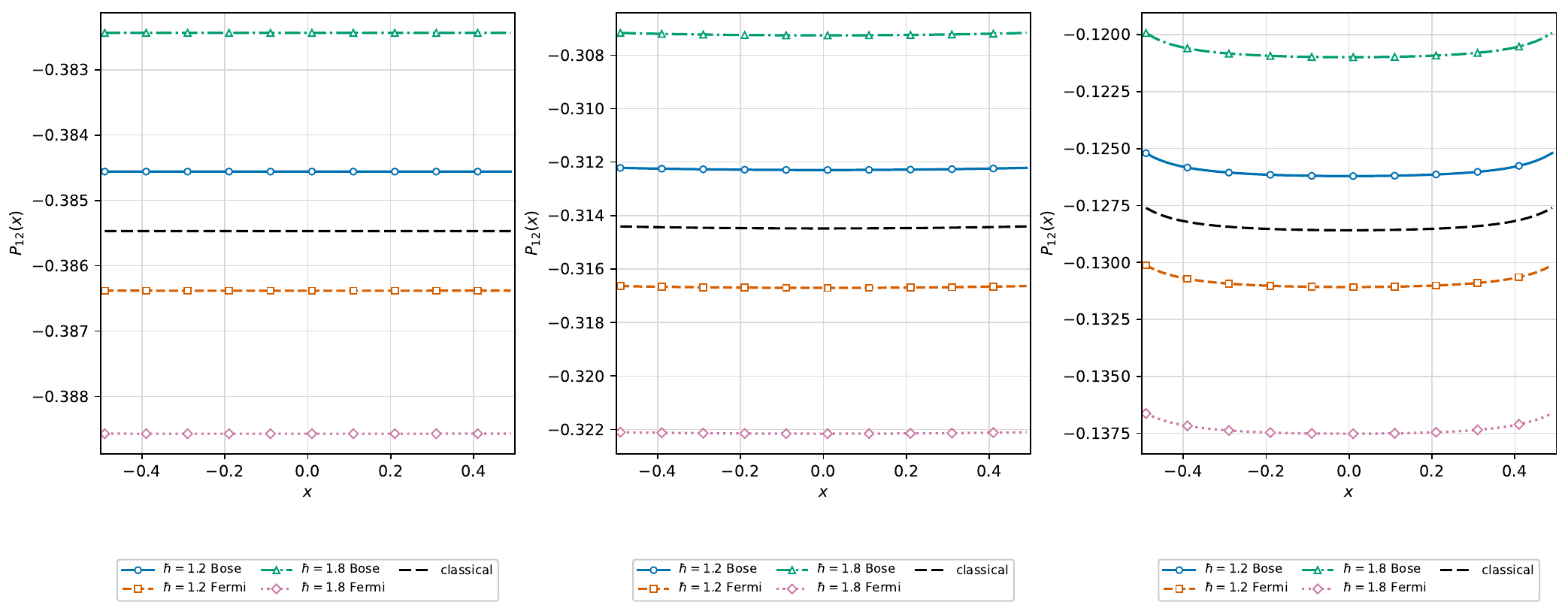}
    \end{minipage}
    \caption{Couette flow: (a) $u_2(x_1)$, (b) $\theta(x_1)$, and (c) $P_{12}(x_1)$ for $\Kn=10,1,0.1$ from left to right. Circles (Bose) and squares (Fermi) indicate $\hbar=1.2$; triangles (Bose) and diamonds (Fermi) indicate $\hbar=1.8$. Black dashed curves show the classical Boltzmann reference. The plotted coordinate $x$ denotes $x_1$. Vertical scales differ between columns.}
    \label{fig:couette-profiles}
\end{figure}
\FloatBarrier

\paragraph{Convergence and computational cost}
We compare the effective Jacobian used in Algorithm~\ref{alg:modified_newton} with a BGK approximation at $\Kn=1$ and $0.1$.
At each outer step $k$, the latter is frozen as
\begin{equation}
    \widetilde{\mathcal L}_{\mathrm{BGK},i}^{(k)}[g]
    =\nu_i^{(k)}
      \bigl(\Pi_{\mathcal M_{q,i}^{(k)}}g-g\bigr),
    \qquad
    \nu_i^{(k)}=\max\{\nu_{\min},\rho_i^{(k)}\},
    \qquad \nu_{\min}=10^{-3},
    \label{eq:bgk-jacobian-comparison}
\end{equation}
where $\mathcal M_{q,i}^{(k)}$ is the local moment-matched equilibrium and $\rho_i^{(k)}$ the density.
The five-moment tangent projection $\Pi_{\mathcal M}$ maps $g$ onto the tangent space of quantum equilibria at $\mathcal M$, preserving its mass, momentum, and energy moments.
Only the approximate collision Jacobian is changed: both methods retain the full QBE residual and all other steps of Algorithm~\ref{alg:modified_newton}.

% Temporarily omitted convergence plot; retained for restoration.
% \begin{figure}[pos=!htbp]
%     \centering
%     \includegraphics[width=0.98\linewidth]{figs/Couette/CouetteFlow_Iterations_Kn1.jpg}
%     \caption{Couette flow at $\Kn=1$: full QBE outer residual norms $\mathcal E_{k,0}$ (left) and relative inner residuals $\mathcal E_{k,l}/\mathcal E_{k,0}$ for the effective Jacobian (right). Rows correspond to $\hbar=1.2$ and $1.8$. Blue and orange denote Bose and Fermi statistics. In the left panels, solid circles and dashed squares distinguish the effective and BGK Jacobians; in the right panels, Newton 1--3 denote the inner solves at $k=0,1,2$.}
%     \label{fig:couette-iterations-kn1}
% \end{figure}

% Temporarily omitted convergence plot; retained for restoration.
% \begin{figure}[pos=!htbp]
%     \centering
%     \includegraphics[width=0.98\linewidth]{figs/Couette/CouetteFlow_Iterations_Kn0p1.jpg}
%     \caption{Couette flow at $\Kn=0.1$, with the same residual definitions, panel arrangement, and curve conventions as Figure~\ref{fig:couette-iterations-kn1}.}
%     \label{fig:couette-iterations-kn01}
% \end{figure}

Table~\ref{tab:newton-couette-iteration-timing} reports iteration counts and computational cost.
Here $T_{\mathrm{tot}}$ is the total solver wall time, including initialization, full nonlinear collision assembly, approximate-Jacobian construction and inner solves.
At $\Kn=1$, the effective Jacobian requires $3$--$5$ outer iterations, compared with $7$ for the BGK Jacobian; at $\Kn=0.1$, the counts are $3$--$5$ and $12$--$13$.
The cumulative inner counts are also lower for the effective Jacobian.
Each outer iteration requires one full QBE collision evaluation, and full nonlinear collision assembly accounts for more than $97\%$ of $T_{\mathrm{tot}}$ in every case.
Reducing full collision evaluations therefore accounts for most of the saving: the BGK runs take $1.33$--$1.99$ times as long at $\Kn=1$ and $2.16$--$3.24$ times as long at $\Kn=0.1$.

\begin{table}[pos=!htbp]
\centering
\caption{Couette flow: iteration counts and computational cost for the effective and BGK Jacobians. Outer counts completed outer iterations, each requiring one full QBE collision evaluation; Inner is the cumulative number of inner iterations. The percentage $p_Q$ is the share of $T_{\mathrm{tot}}$ spent assembling the full nonlinear collision term; $T_{\mathrm{tot}}$ is reported in seconds, rounded to the nearest integer. All runs use the tolerances in \eqref{eq:error_settings} and reach the outer stopping criterion.}
\label{tab:newton-couette-iteration-timing}
\footnotesize
\setlength{\tabcolsep}{4.5pt}
\renewcommand{\arraystretch}{1.08}
\begin{tabular}{@{}ccc rrrr rrrr@{}}
\toprule
 & & & \multicolumn{4}{c}{Effective Jacobian} & \multicolumn{4}{c}{BGK Jacobian} \\
\cmidrule(lr){4-7}\cmidrule(l){8-11}
$\Kn$ & $\hbar$ & Gas & Outer & Inner & $p_Q$ (\%) & $T_{\mathrm{tot}}$ (s) & Outer & Inner & $p_Q$ (\%) & $T_{\mathrm{tot}}$ (s) \\
\midrule
$1$ & $1.2$ & Bose & 3 & 19 & 98.68 & 4780 & 7 & 35 & 98.92 & 9530 \\
 &  & Fermi & 3 & 18 & 98.74 & 4785 & 7 & 36 & 98.89 & 9531 \\
 & $1.8$ & Bose & 5 & 31 & 98.59 & 7196 & 7 & 37 & 98.86 & 9554 \\
 &  & Fermi & 4 & 24 & 98.67 & 5993 & 7 & 36 & 98.88 & 9553 \\
\midrule
$0.1$ & $1.2$ & Bose & 4 & 51 & 97.13 & 6086 & 12 & 125 & 97.55 & 15749 \\
 &  & Fermi & 3 & 39 & 97.23 & 4862 & 12 & 129 & 97.46 & 15749 \\
 & $1.8$ & Bose & 5 & 63 & 97.05 & 7307 & 12 & 131 & 97.40 & 15748 \\
 &  & Fermi & 4 & 50 & 97.19 & 6092 & 13 & 141 & 97.41 & 16960 \\
\bottomrule
\end{tabular}
\end{table}

\FloatBarrier

\subsection{Fourier Flow}
\label{subsec:fourier}
We consider steady heat conduction between stationary plates with wall data
\begin{equation}
    \bm u_L=\bm u_R=\bm 0,\qquad
    \theta_L=1,\qquad \theta_R=3.
    \label{eq:fourier-wall-parameters}
\end{equation}
The spatial domain, total mass, collision kernel, cutoff, grid, and parameter sweep are those of Section~\ref{subsec:couette}, as are the boundary treatment and stopping tolerances \eqref{eq:error_settings}.

The initial state is given by \eqref{eq:IV} with $\bm u^{(0)}=\bm 0$ and spatially varying parameters $z^{(0)}(x_i),T^{(0)}(x_i)$.
These are determined at each interior cell centre from \eqref{eq:fugality_and_temperature} with
\begin{equation*}
    \rho^{(0)}=1,\qquad
    \theta^{(0)}(x_i)
    =\frac{x_R-x_i}{x_R-x_L}\theta_L
     +\frac{x_i-x_L}{x_R-x_L}\theta_R.
\end{equation*}

\paragraph{Steady profiles}
Figure~\ref{fig:fourier-profiles} shows the kinetic temperature, normal heat flux, and relative normal-stress deviation $(P_{11}-p)/p$, with $P_{11}$ and $p=\rho\theta$ defined in \eqref{eq:momentums}.

As $\Kn$ decreases, the bulk temperature gradient increases and the temperature jumps at the plates decrease.
The Bose and Fermi temperature profiles remain close and cross within the channel, so neither statistic gives a uniformly higher temperature.
Their heat fluxes are more clearly separated: relative to the classical Boltzmann reference, Bose statistics reduce $|q_1|$ and Fermi statistics increase it, with a larger separation at $\hbar=1.8$.
In every case $q_1<0$, corresponding to heat transport from the hot right plate to the cold left plate.
The heat-flux magnitude decreases as $\Kn$ decreases over the range considered. The spatial variation of $q_1$ is a spatial-discretization error that depends on the mesh size and is more pronounced at smaller $\Kn$.

\begin{figure}[pos=!htbp]
    \centering
    \begin{minipage}{0.98\linewidth}
        \small (a) Kinetic temperature\par
        \includegraphics[width=\linewidth]{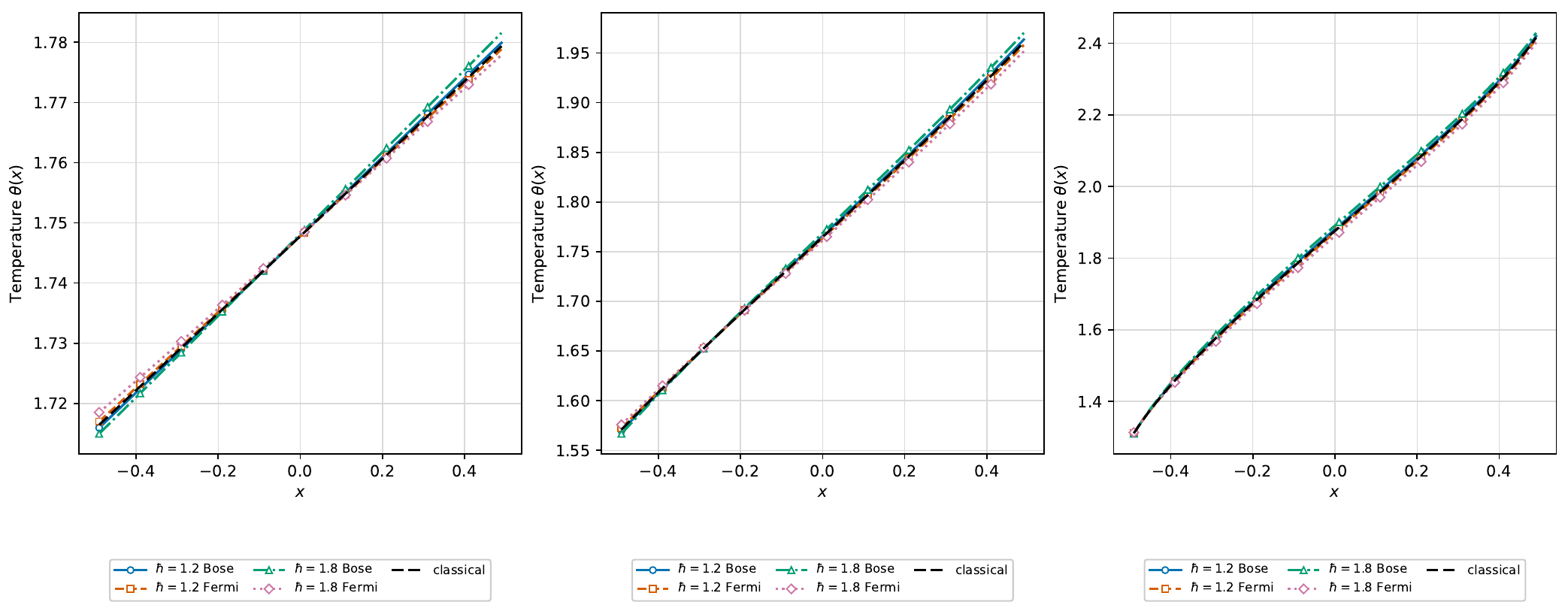}\par\vspace{2mm}
        \small (b) Normal heat flux\par
        \includegraphics[width=\linewidth]{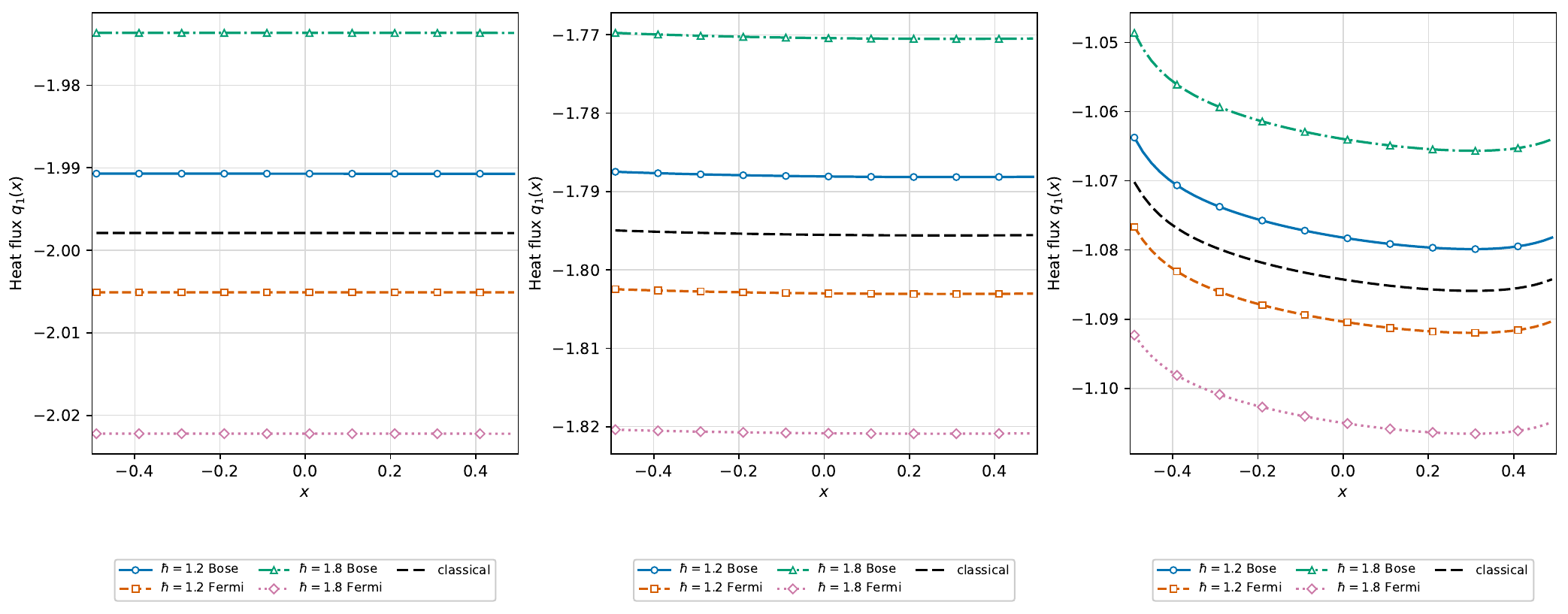}\par\vspace{2mm}
        \small (c) Relative normal-stress deviation\par
        \includegraphics[width=\linewidth]{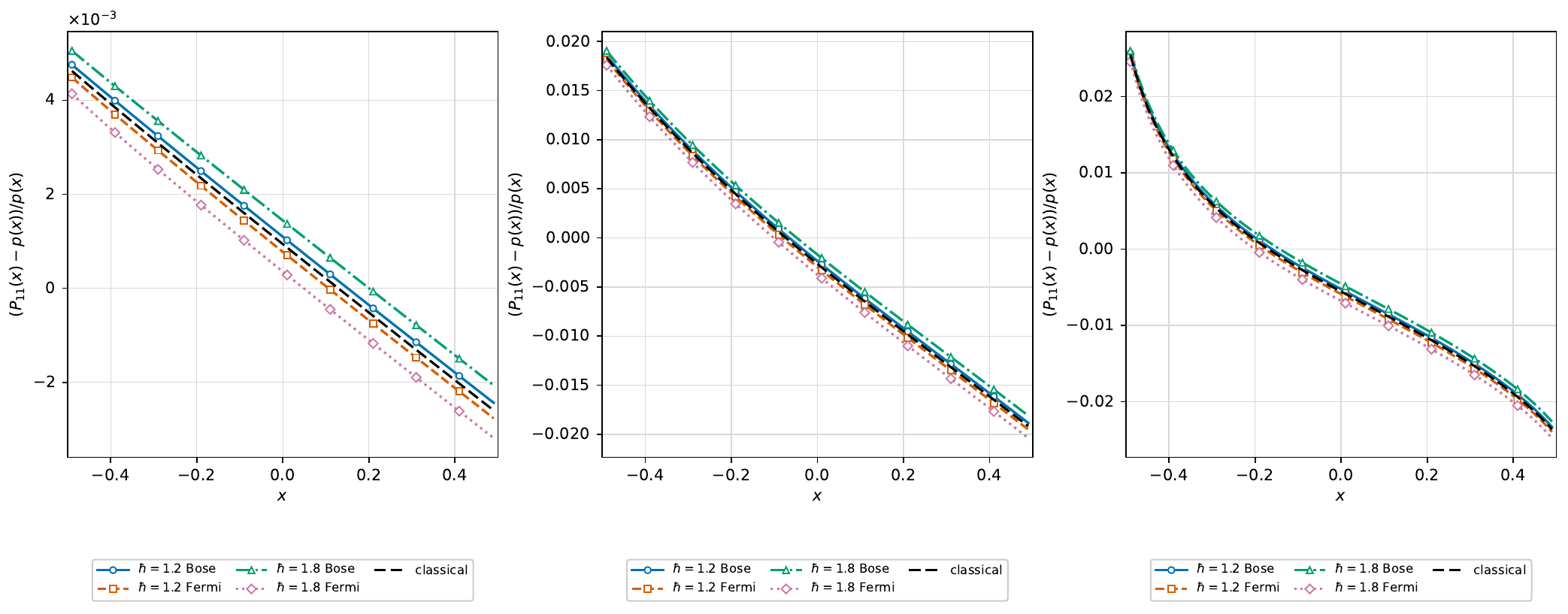}
    \end{minipage}
    \caption{Fourier flow: (a) $\theta(x_1)$, (b) $q_1(x_1)$, and (c) $(P_{11}(x_1)-p(x_1))/p(x_1)$, with $p=\rho\theta$ as in \eqref{eq:momentums}, for $\Kn=10,1,0.1$ from left to right. Quantum cases and classical reference follow the curve conventions of Figure~\ref{fig:couette-profiles}; the plotted coordinate is $x=x_1$. Vertical scales differ between columns.}
    \label{fig:fourier-profiles}
\end{figure}
\FloatBarrier

% The heat flux is nearly uniform for $\Kn=10$ and $1$, but retains a visible spatial variation at $\Kn=0.1$.
% For the continuous steady problem with zero bulk velocity, energy conservation gives $\partial_{x_1}q_1=0$.
% The remaining variation is therefore a numerical defect in this cell-centred diagnostic; convergence of the nonlinear residual does not by itself establish spatial accuracy of the heat flux.

The relative normal-stress deviation is positive near the cold plate and negative near the hot plate, showing a departure from the isotropic equilibrium pressure tensor.
Over the range considered, the wall-side deviations grow as $\Kn$ decreases, with more pronounced curvature at $\Kn=0.1$.
The Bose curves lie above the classical reference and the Fermi curves below it, with larger shifts at $\hbar=1.8$.
This spatial variation is compatible with $\partial_{x_1}P_{11}=0$ from the steady normal-momentum balance \eqref{eq:Quantum_Euler}, since the scalar pressure $p$ need not be constant.

\paragraph{Convergence and computational cost}
We compare the effective and BGK Jacobians under the same protocol as Section~\ref{subsec:couette}.

% Temporarily omitted convergence plot; retained for restoration.
% \begin{figure}[pos=!htbp]
%     \centering
%     \includegraphics[width=0.98\linewidth]{figs/Fourier/FourierFlow_Iterations_Kn1.jpg}
%     \caption{Fourier flow at $\Kn=1$: full QBE outer residual norms for the effective and BGK Jacobians (left), and relative inner residuals for the first three outer steps with the effective Jacobian (right). Rows correspond to $\hbar=1.2$ and $1.8$; residual definitions and curve conventions are as in Figure~\ref{fig:couette-iterations-kn1}.}
%     \label{fig:fourier-iterations-kn1}
% \end{figure}

% Temporarily omitted convergence plot; retained for restoration.
% \begin{figure}[pos=!htbp]
%     \centering
%     \includegraphics[width=0.98\linewidth]{figs/Fourier/FourierFlow_Iterations_Kn0p1.jpg}
%     \caption{Fourier flow at $\Kn=0.1$, with the same panel arrangement and curve conventions as Figure~\ref{fig:fourier-iterations-kn1}.}
%     \label{fig:fourier-iterations-kn01}
% \end{figure}

Table~\ref{tab:newton-fourier-iteration-timing} gives $4$ outer iterations for the effective Jacobian at $\Kn=1$, compared with $8$--$9$ for the BGK Jacobian; at $\Kn=0.1$, the counts are $5$ and $17$.
The cumulative inner counts are also lower for the effective Jacobian in every case.
Full nonlinear collision assembly accounts for more than $95\%$ of $T_{\mathrm{tot}}$.
The BGK runs take $1.86$--$2.14$ times as long at $\Kn=1$ and $3.11$--$3.25$ times as long at $\Kn=0.1$, so the reduction in total wall time is greater at the smaller Knudsen number.

\begin{table}[pos=!htbp]
\centering
\caption{Fourier flow: iteration counts and computational cost for the effective and BGK Jacobians. Outer, Inner, $p_Q$ and $T_{\mathrm{tot}}$ follow Table~\ref{tab:newton-couette-iteration-timing}; all runs use the tolerances in \eqref{eq:error_settings} and reach the outer stopping criterion.}
\label{tab:newton-fourier-iteration-timing}
\footnotesize
\setlength{\tabcolsep}{4.5pt}
\renewcommand{\arraystretch}{1.08}
\begin{tabular}{@{}ccc rrrr rrrr@{}}
\toprule
 & & & \multicolumn{4}{c}{Effective Jacobian} & \multicolumn{4}{c}{BGK Jacobian} \\
\cmidrule(lr){4-7}\cmidrule(l){8-11}
$\Kn$ & $\hbar$ & Gas & Outer & Inner & $p_Q$ (\%) & $T_{\mathrm{tot}}$ (s) & Outer & Inner & $p_Q$ (\%) & $T_{\mathrm{tot}}$ (s) \\
\midrule
$1$ & $1.2$ & Bose & 4 & 30 & 97.85 & 5563 & 8 & 48 & 98.44 & 10347 \\
 &  & Fermi & 4 & 28 & 98.01 & 5553 & 9 & 54 & 98.39 & 11562 \\
 & $1.8$ & Bose & 4 & 29 & 98.28 & 5527 & 8 & 49 & 98.37 & 10538 \\
 &  & Fermi & 4 & 29 & 98.28 & 5529 & 9 & 54 & 98.40 & 11823 \\
\midrule
$0.1$ & $1.2$ & Bose & 5 & 66 & 95.54 & 6925 & 17 & 204 & 95.83 & 21522 \\
 &  & Fermi & 5 & 66 & 95.54 & 6917 & 17 & 207 & 95.74 & 21789 \\
 & $1.8$ & Bose & 5 & 68 & 96.59 & 6861 & 17 & 203 & 96.22 & 22307 \\
 &  & Fermi & 5 & 72 & 96.38 & 6883 & 17 & 210 & 96.07 & 22319 \\
\bottomrule
\end{tabular}
\end{table}

\FloatBarrier

\section{Conclusion}\label{sec:conclusion}
%Para1: Summary
We constructed an effective linear operator $\widetilde{\mathcal{L}}_q$ for the quantum Boltzmann equation that preserves mass, momentum and energy and satisfies a linearized $H$-theorem.
Its evaluation is substantially cheaper than \texttt{FullQBE}, and the operator tests show closer agreement with the exact linearization at smaller $\hbar$.
As a collision model, it incurs modeling error but gives smaller errors than the quantum BGK model in the homogeneous relaxation tests (Section~\ref{sec:spatial_homo}).
As an approximate collision Jacobian, it reduces the number of full collision evaluations in the modified Newton solver for steady Couette and Fourier flows (Sections~\ref{subsec:couette} and~\ref{subsec:fourier}).
The solver retains the full nonlinear residual, so the Jacobian approximation leaves the target steady equation unchanged.
%Para2: Meaning

In the fast spectral framework used here, the quantum and classical collision operators have the same asymptotic complexity, but the additional cubic terms typically increase the constant factor for the quantum operator.
The effective linear approximation reduces this cost and makes finer velocity grids (larger $N_v$) more affordable for simulations with the linear model.
This offers a practical tool for studying quantum kinetic relaxation where the modeling error is acceptable, particularly for smaller values of the rescaled Planck constant $\hbar$.
%Para3: Future works (FW1:Bose-Einstein Condense, FW2:Effective FullQBE) 

Two directions merit further investigation.
One is to extend the present approach to Bose gases with stronger quantum effects, including regimes in which Bose--Einstein condensation may occur.
The other is to explore operator-learning approximations to the full nonlinear collision operator $\mathcal{Q}_q$, whose evaluation dominates the computational cost of the modified Newton solver in the reported tests.

\printcredits

\section*{Acknowledgments}
Zhenning Cai was partially supported by the Academic Research Fund of
the Ministry of Education of Singapore (Grant No. A-8002392-00-00).

\bibliographystyle{cas-model2-names}

% Keep the final reference on the bibliography page.
\renewcommand{\bibfont}{\fontsize{8pt}{9.6pt}\selectfont}
\bibliography{cas-refs}



\end{document}